\documentclass[11pt,reqno]{amsart}

\newtheorem{thm}{Theorem}[subsection]
\newtheorem{cor}[thm]{Corollary}
\newtheorem{lem}[thm]{Lemma}
\newtheorem{pro}[thm]{Proposition}

\theoremstyle{remark}
\newtheorem{rem}[thm]{Remark}
\newtheorem{irem}{Remark}
\newtheorem*{ack}{Acknowledgements}
\newtheorem*{nota}{Notation}

\theoremstyle{definition}
\newtheorem{defi}[thm]{Definition}
\newtheorem{ex}[thm]{Example}

\usepackage[mathscr]{eucal}
\usepackage{amscd}
\usepackage{dsfont}
\usepackage{stmaryrd}
\usepackage{amssymb}
\usepackage{enumerate}
\usepackage{stmaryrd}
\usepackage{mathdots}
\SetSymbolFont{stmry}{bold}{U}{stmry}{m}{n}

\DeclareMathOperator{\tr}{tr}
\DeclareMathOperator{\res}{Res}
\DeclareMathOperator{\obj}{Obj}
\let \hom \relax
\DeclareMathOperator{\hom}{Hom}
\DeclareMathOperator{\spec}{Spec}

\DeclareFontFamily{OT1}{pzc}{}
\DeclareFontShape{OT1}{pzc}{m}{it}{<-> s * [1.20] pzcmi7t}{}
\DeclareMathAlphabet{\mathpzc}{OT1}{pzc}{m}{it}

\newcommand{\sym}{\mathpzc{Sym}}
\newcommand{\shgs}{\sym \llbracket (H^G)^\ast \rrbracket}
\newcommand{\skgs}{\sym \llbracket (K^G)^\ast \rrbracket}
\newcommand{\snhs}{\sym^n H^\ast}
\newcommand{\snhgs}{\sym^n (H^G)^\ast}

\newcommand{\hgn}{(H^G)^{\otimes n}}

\newcommand{\hgs}{[(H^G)^\ast]}
\newcommand{\hgsn}{[(H^G)^\ast]^{\otimes n}}
\newcommand{\pn}{\phi^{\otimes n}}
\newcommand{\br}{\mathpzc{Br}}
\newcommand{\brhs}{\mathpzc{Br} \llbracket H^\ast \rrbracket}
\newcommand{\brh}{\mathpzc{Br} \llbracket H \rrbracket}
\newcommand{\brk}{\mathpzc{Br} \llbracket K \rrbracket}
\newcommand{\brks}{\mathpzc{Br} \llbracket K^\ast \rrbracket}
\newcommand{\brnh}{\mathpzc{Br}^n H}
\newcommand{\brnhs}{\mathpzc{Br}^n H^\ast}
\newcommand{\brmh}{\mathpzc{Br}^m H}

\newcommand{\brnk}{\mathpzc{Br}^n K}

\newcommand{\brmk}{\mathpzc{Br}^m K}

\newcommand{\brnch}{\mathpzc{Br}^n_C H} 
 
\newcommand{\brngch}{\mathpzc{Br}^n_{g \cdot C} H} 
 
\newcommand{\hsn}{(H^\ast)^{\otimes n}}

\newcommand{\hn}{H^{\otimes n}}
\newcommand{\hm}{H^{\otimes m}}
\newcommand{\hk}{H^{\otimes k}}
\newcommand{\kn}{K^{\otimes n}}
\newcommand{\zz}{\mathbb{Z}/2\mathbb{Z}}
\newcommand{\hs}{\llbracket H^\ast \rrbracket}
\newcommand{\symnhs}{\mathpzc{Sym}^n H^\ast}
\newcommand{\symhs}{\mathpzc{Sym}\llbracket H^\ast \rrbracket}
\newcommand{\symks}{\mathpzc{Sym}\llbracket K^\ast \rrbracket}
\newcommand{\Th}{\mathpzc{T} \llbracket H \rrbracket}
\newcommand{\ths}{\mathpzc{T} \llbracket H^\ast \rrbracket}
\newcommand{\sn}{\mathpzc{S}_n}
\newcommand{\s}{\mathpzc{S}}
\newcommand{\dcg}{D(k[G])}
\newcommand{\md}{\mathbf{Mod}}
\newcommand{\B}{\mathpzc{B}}
\newcommand{\bn}{\B_n}
\newcommand{\g}{\mathcal{G}}
\newcommand{\gn}{\mathcal{G}^n}
\newcommand{\C}{\mathbb{C}}
\newcommand{\bs}{\boldsymbol}
\newcommand{\bm}{\boldsymbol m}
\newcommand{\bg}{\boldsymbol \gamma}
\DeclareMathOperator{\bc}{\mathbf{C}}
\DeclareMathOperator{\bcf}{\bc_{fin}}
\DeclareMathOperator{\vc}{\mathbf{Vect}_\C}
\newcommand{\igs}{(i^G)^\ast}
\newcommand{\ies}{(i_e)^\ast}
\newcommand{\xh}{\mathfrak{X}H}
\newcommand{\mgzb}[1]{\overline{M}^G_{0,#1}}
\newcommand{\mzb}[1]{\overline{M}_{0,#1}}
\DeclareMathOperator{\sttilde}{\widetilde{st}}
\newcommand{\stt}[1][]{\sttilde_{#1}}
\DeclareMathOperator{\sthat}{\widehat{st}}
\newcommand{\sth}[1][]{\sthat_{#1}}
\DeclareMathOperator*{\st}{st}
\renewcommand{\bar}{\overline} 
\renewcommand{\b}[1]{\boldsymbol{#1}}
\newcommand{\bcdot}{\boldsymbol \cdot}
\newcommand{\I}{\mathcal{I}}
\DeclareMathOperator{\Span}{Span}
\newcommand{\CR}[1]{H^\ast_{\text{CR}} \negthinspace \left(#1 \right)}
\newcommand{\Hs}[1]{H^\ast \negthinspace \left( #1 \right)}
\begin{document}

\title{$G$-Frobenius Manifolds}
\author{Byeongho Lee}

\begin{abstract}
 The goal of this paper is to introduce the notion of $G$-Frobenius manifolds for
any finite group $G$. This work is motivated by the fact that
any $G$-Frobenius algebra yields an ordinary Frobenius algebra
by taking its $G$-invariants.
We generalize this on the level of Frobenius manifolds.

To define a $G$-Frobenius manifold as a braided-commutative generalization 
of the ordinary commutative Frobenius manifold, we develop the theory
of $G$-braided spaces. 
These are defined as $G$-graded $G$-modules with certain braided-commutative
``rings of functions'', generalizing the commutative rings of power series on 
ordinary vector spaces.

As the genus zero part of any ordinary cohomological field theory 
of Kontsevich-Manin contains a Frobenius manifold, we show that
any $G$-cohomological field theory defined by Jarvis-Kaufmann-Kimura contains
a $G$-Frobenius manifold up to a rescaling of its metric.

Finally, we specialize to the case of $G = \zz$ and prove the
structure theorem for (pre-)$\zz$-Frobenius manifolds.
We also construct an example of a $\zz$-Frobenius manifold using this theorem,
that arises in singularity theory in the hypothetical context of orbifolding.
\end{abstract}

\maketitle

\section{Introduction}
\label{sec:intro}

There are several statements in the philosophy of mirror symmetry.
One of them involves Frobenius manifolds.
Namely, the Frobenius manifolds of a mirror pair should be isomorphic.

Quantum cohomology ring of a smooth projective variety $X$
is an example of a Frobenius manifold in the A-model.
It encodes the genus zero part of the Gromov-Witten invariants of $X$
on its cohomology ring.

According to Costello\cite{costello}, genus $g$ GW-invariants of $X$
can be written in terms of the genus zero GW-invariants of $[X^{g+1}/S_{g+1}]$,
the $(g+1)$th symmetric power of $X$.
The latter can be encoded as a Frobenius manifold structure on 
its Chen-Ruan orbifold cohomology ring\cite{oqc, croc}.

Now suppose $X$ has
a finite group $G$-action on it, 
and let $[X/G]$ be the global quotient orbifold.
On the level of Frobenius algebras, Fantechi-G\"ottsche constructed a 
certain braided-commutative ring $\Hs{X,G}$ 
with $G$-action, now called a stringy cohomology ring,
that yields $\CR{[X/G]}$
 by taking the $G$-invariants\cite{fg}.
$\Hs{X,G}$ also contains the cohomology of $X$ as a subring.

$\Hs{X,G}$ is an example of the structure called a $G$-Frobenius algebra 
(also known as crossed $G$-algebra, or Turaev algebra)\cite{orbifolding, ms, turaev}.
Kaufmann generalized the construction in \cite{fg} and showed that
one can produce a certain type of $G$-Frobenius algebras
starting from a Jacobian Frobenius algebra\cite{orbifolding}.
The resulting $G$-Frobenius algebra contains the original Frobenius algebra 
as a subalgebra called the untwisted sector.
In particular, he proved that there is an essentially unique $n$th symmetric 
power, an $S_n$-Frobenius algebra, of each Jacobian Frobenius algebra\cite{sq}.
Note that this result works for both A- and B-model sides of mirror symmetry.

One can then ask if we can generalize his results on the level of 
Frobenius manifolds.
Namely, we want a notion of $G$-Frobenius manifold structure
on each $G$-Frobenius algebra, that contains an ordinary Frobenius
manifold on the untwisted sector and yields another Frobenius manifold
on the $G$-invariants of the underlying $G$-Frobenius algebra.
Once we have this, we should ask if we can construct a $G$-Frobenius
manifold starting from an ordinary Frobenius manifold and the $G$-Frobenius algebra
that extends the underlying Frobenius algebra of the given Frobenius manifold.
In particular, we should investigate if we can construct a symmetric power
of each Frobenius manifold, and how unique it is.

Once the answers to these questions are obtained,
we will be able to apply them to the GW-theory of symmetric powers of $X$.
Starting from the quantum cohomology of $X$, 
we build the $S_n$-Frobenius manifold of $X^n$, and taking its $S_n$-invariants
will yield the orbifold quantum cohomology of the symmetric power.
By Costello's formula, we will essentially know all the higher genus GW-invariants.

We start to answer this series of questions in this paper.
Namely, we define a notion of $G$-Frobenius manifold for each finite group
$G$ and find some examples.

In the definition of ordinary Frobenius manifolds, certain symmetric multilinear
forms called correlation functions play an important role.
Each of these is given as the homogeneous degree $n$ part of the potential of
the given Frobenius manifold, using the well-known identification
between homogeneous polynomials and symmetric multilinear forms.
The underlying Frobenius algebra is obtained by truncating the potential at degree $3$.

On the other hand, the role of $n$-linear forms that are 
invariant under the braid group $B_n$ plays a similar role in the
theory of $G$-Frobenius algebras for $n=2$ and $3$. In fact, a $G$-Frobenius algebra
is a monoid object in the braided monoidal category of $G$-graded $G$-modules, 
$\dcg$-$\md$.
This motivates us to study \emph{braided tensors} in the category.
See Definition~\ref{defi:bten}.  

Once we identify the ring of power series on a finite dimensional vector space as
the ring of symmetric multilinear forms on it,
we are tempted to define a \emph{$G$-braided space} as a finite dimensional
object in $\dcg$-$\md$ with a suitable ring structure on braided multilinear forms as
the ``ring of functions'', following a common practice in the field of 
noncommutative geometry. We do this in Section~\ref{sec:gbraided}. 
We first develop the theory of braided tensors
in certain braided monoidal categories, and specialize to $\dcg$-$\md$.
When $G$ is trivial, we recover the notion
of finite dimensional vector spaces with their commutative ring of
power series.
Theorem~\ref{thm:ringmap} is the main result in this section.

After reviewing the notions of Frobenius manifolds and $G$-Frobenius algebras,
We first define the notion of pre-$G$-Frobenius manifolds in Section~\ref{sec:gfm}.
They are certain $G$-braided spaces containing two Frobenius manifolds
as their subspaces on the untwisted sector and the $G$-invariants. 
If they also contain a $G$-Frobenius algebra,
in a similar way as an ordinary Frobenius manifold contains a Frobenius algebra,
then they are called $G$-Frobenius manifolds.
Definition~\ref{defi:gfm} gives the precise descriptions.

We show that any $G$-Cohomological field theory ($G$-CohFT) of Jarvis-Kaufmann-Kimura\cite{jkk}
contains a pre-$G$-Frobenius manifold structure on its state space
in Section~\ref{sec:corrgcohft}. 
In fact, it also contains
a $G$-Frobenius algebra up to a rescaling of its metric.
To prove this result, we introduce two types of correlation functions
that one can obtain out of any $G$-CohFT: a symmetric one and a
$G$-braided one. We use these and relevant results in \cite{jkk} to prove that 
one can obtain a $G$-braided potential that yields a
symmetric potential for a Frobenius manifold structure on its subspace of $G$-invariants.
See Theorem~\ref{thm:cor} and Remark~\ref{rem:scaling}.

In singularity theory, $A_{2n-3}$ singularities have natural $\zz$-actions.
In physics literature regarding Landau-Ginzburg B-models, 
the Frobenius manifolds for $D_n$ are 
obtained from $A_{2n-3}$ using the procedure of orb\-ifolding\cite{dvv},
but there is no satisfactory mathematical theory about it yet,
except for the level of Frobenius algebras\cite{orbifolding}.
We construct the examples of $\zz$-Frobenius manifolds for $A_{2n-3}$
as possible partial ingredients for such a theory in Section~\ref{sec:example}. 
The first step is to prove the Structure Theorem~\ref{thm:structure}
for any pre-$\zz$-Frobenius manifold. 
After reviewing known facts about the Frobenius manifold and $G$-Frobenius algebra
structures for the singularities $A_{2n-3}$ and $D_n$,
we use this theorem to construct pre-$\zz$-Frobenius manifolds for $A_{2n-3}$.
In fact, it turns out that these are $\zz$-Frobenius manifolds.
Section~\ref{ssec:zzfm} contains the details.

We conclude our paper with a brief discussion of a possible development of the
concept of $G$-braided spaces. In particular, we discuss about a differentiation rule
on the ring of braided multilinear forms.

We expect that other important structures on $G$-Frobenius manifolds will emerge 
in the process of answering the questions above.
Rather than using {\it a priori} constructions,
we will take the approach of picking up necessary 
structures while analyzing more examples and
working on relevant theories.
Although we gave a minimal definition for  $G$-Frobenius manifolds,
it is strong enough for $G=\zz$ in the sense that we see in Theorem~\ref{thm:structure}.

\begin{irem}
  This work is also a starting point for answering 
  one of the Turaev's questions\cite[Appendix 4, No. 2]{turaev}.
\end{irem}

\begin{irem}
  Matsumura-Shapiro considered a different version of
  $G$-braided geometry on $G$-graded $G$-modules using 
  Majid's theory of braided calculus\cite[Chapter 10]{majid}
  to define a different notion of $G$-Frobenius structure on
  them\cite{matsu}.
\end{irem}

\begin{ack}
  We thank Professor R. Kaufmann for all his guidance.
 We are also grateful to Professors A. Libgober and H. Tseng
 for useful comments on this work.
 We also thank Professors C. Cho, R. Cavalieri, T. Jarvis, T. Matsumura, D. Patel, and S. Yeung
 for helpful discussions.
 Conversations with Y. Tsumura and A. Jackson were also helpful.

 A significant portion of this work was supported by Bilsland Dissertation Fellowship
 from the Graduate School of Purdue University.
\end{ack}

\begin{nota}
 Throughout this paper, 
$G$ is a finite group,
$k$ a field of characteristic zero.
We use the Einstein convention, namely, we sum over the
repeated upper and lower indices.
The identity in $G$ is denoted by $e$.
We do not consider the trace axiom of $G$-Frobenius algebras in this paper.
We use boldface letters to denote elements of $G^n$.
For example, $\bm$ means $(m_1, \dots, m_n) \in G^n$.
The identity in $G$ is denoted by $e$.
\end{nota}

\section{$G$-braided Spaces}
\label{sec:gbraided}

\subsection{Functions on Vector Spaces}
\label{ssub:functions}

 We review some well-known facts about ordinary vector spaces.
The goal of this section is to fix notations and emphasize the viewpoint
that will be useful for the generalization we have in mind.

Let $H$ be a finite dimensional vector space over $k$. 
We take the ring of power series $k\hs$ to be its ring of functions.
This is required for the definition of Frobenius manifold.
See Section~\ref{sec:FM}.

It is well known that $k\hs_n$, the homogeneous polynomials of degree $n$,
can be identified with $\symnhs := [\hsn]^{S_n}$, the subspace of invariants
under the natural $S_n$ action, also known as the symmetric $n$-linear forms.
Hence we have isomorphisms of vector spaces
\[
  k\hs \cong \prod_{n=0}^{\infty} \symnhs := \symhs \, .
\]

Also, observe that there is another characterization of $\snhs$, considering
the $S_n$ action on $\hn$. Namely,
\[
\snhs = \{ x \in \hsn |
  x(v) = x(\sigma \cdot v) \ \text{for any} \ \sigma \in S_n \ \text{and} \ 
v \in \hn \} \, .
\]

Note that $k\hs$ also has the structure of a commutative ring.
We transfer this ring structure to $\symhs$.
Note also that $\symhs$ is a vector subspace 
of the complete tensor algebra $\prod_{n=0}^{\infty} \hsn := \ths$.
The latter also has the natural ring structure of juxtaposition, but 
$\symhs$ is not a subring of $\ths$.

If we consider how these two multiplication rules relate to each other 
and view the ring structure of $\symhs$ intrinsically,
we realize that the following map is in action.

\begin{defi}
  \label{defi:symmetrization}
  For any $n \geq 0$, the \emph{$n$th symmetrization on $H$} is the map
  \[
    \sn : \hsn \to \symnhs\, , \quad x \mapsto \frac{1}{|S_n|} 
    \sum_{\sigma \in S_n} \sigma x \, .
  \]
  Also, the \emph{symmetrization on $H$} is the map
  \[
    \s := \prod_{n=0}^{\infty} \sn : \ths \to \symhs \, .
  \]
\end{defi}

We observe that the ring structure on $\symhs$ is given as follows.
Let $\mu$ be the multiplication map in $\ths$, namely, the juxtaposition.
\begin{pro}
  \label{pro:symmetrization}
 The composition of the following maps
 \[
\symhs \otimes \symhs \hookrightarrow 
\ths \otimes \ths \xrightarrow{\mu}
\ths \xrightarrow{\s}
\symhs
 \]
 is the multiplication map of $\symhs$.
\end{pro}

\begin{rem}
  In view of the identification of $k\hs$ with $\symhs$, 
  we regard $k\hs$ as the ring of
  symmetric multilinear forms.
\end{rem}

The following functorial property is also well known.

\begin{pro}
  Let $\phi : H \to K$ be a linear map of finite dimensional vector spaces.
  Then we have the induced map of algebra with unity
  \[
\phi^\ast : \symks \to \symhs \, .
  \]
  \label{pro:symfunctorial}
\end{pro}

\subsection{Braided Tensors}
\label{sec:braided}

In this section, we introduce the notion of braided tensors in 
certain subcategories of 
$\mathbf{Vect}_{k}$, that are also braided monoidal.
We also define ring structures for them.
Note that we work in the covariant setting. 
The dual setting will be considered when we specialize to the category of $G$-graded 
$G$-modules in Section~\ref{ssec:gbrsp}.

We work with a subcategory $\mathbf{C}$ of $\mathbf{Vect}_{k}$,
and fix a braided monoidal structure 
$(\mathbf{C}, \otimes, \mathbf{1} = k, \Phi)$. 
Here $\otimes$ is the usual tensor product of vector spaces,
but $\Phi$ may be different from the usual braided (actually, symmetric)
monoidal structure of $\mathbf{Vect}_k$.
We may simply use $\mathbf{C}$ understanding its braided monoidal structure.

We first generalize the notion of symmetric tensors on vector spaces.
For any $H \in \bc$, note that the braid group $B_n$ acts
on $H^{\otimes n}$ canonically.
Let $\bcf$ be the full subcategory of finite dimensional objects.

\begin{defi}
For each $H \in \bcf$, 
  set 
  \[
    \brnh := 
    (\hn)^{B_n} \, ,
  \]
  the subspace of $B_n$-invariant $n$-tensors.
  Its elements are called the \emph{braided $n$-tensors} on $H$.
Also, set 
\[
  \brh
  := \prod_{n=0}^{\infty} \brnh  \subset \Th \, .
\]
 Its elements are called the \emph{braided tensors} on $H$. 
 \label{defi:bten}
\end{defi}

To prove that braided tensors are functorial,
we need the following lemma.

\begin{lem}
Let $\phi : H \to K$ be a morphism in $\bcf$. 
Then for any $b \in B_n$, we have the following commutative diagram.
\[
  \begin{CD}
\hn @>\pn>> \kn \\
@VVbV       @VVbV \\
\hn @>\pn>> \kn
  \end{CD}
\]
  \label{lem:commute}  
\end{lem}

\begin{proof}
This follows from the functoriality of the braiding in $\bc$.  
\end{proof}

\begin{rem}
 We regard that $\pn = id_k$  when $n=0$.
\end{rem}

\begin{pro}
 $\br^n(\cdot)$ and $\br \llbracket \cdot \rrbracket$
 are functors $\bcf \to \vc$.
\end{pro}

\begin{proof}
Let $\phi : H \to K$ be a morphism in $\bcf$.
Consider the restriction of the morphism
$\pn : \hn \to \kn$ on $\brnh$.
Let $v \in \brnh$. Then by Lemma~\ref{lem:commute}, we have
\[
  b \cdot \pn(v) = \pn (b \cdot v) = \pn (v) \, .
\]
showing that $\pn (v) \in \brnk$.

Set $\phi_\ast := \prod_{n=0}^{\infty} \pn$.
Then $\phi_\ast$ maps $\brh$ to $\brk$.

If $\psi$ is another morphism in $\bcf$, the equations 
$(\psi \circ \phi)^{\otimes n} =  \psi^{\otimes n} \circ \phi^{\otimes n}$ and
$(\psi \circ \phi)_\ast =  \psi_\ast \circ \phi_\ast$
are inherited from the ones before restricting to braided tensors.

We also observe that $id^{\otimes n} = id$ and $id_\ast = id$.
\end{proof}

  Note that $\brnh$ and $\brh$ are vector spaces, but they need not be
  objects of $\mathbf{C}$.
  We want them to be objects in $\bc$.

\begin{defi}
  $\bc$ is \emph{regular}
  if, for each $H \in \bcf$, $\brnh$ and $\brh$ are objects of $\mathbf{C}$,
  and their braidings are natural in the following sense:
  given another object $W$ of $\mathbf{C}$,
  the braiding map
  \[
    \Phi_{\brnh, W}  : \brnh \otimes W \to W \otimes \brnh
  \]
  is given by the restriction of 
  the braiding map 
  \[
    \Phi_{\hn, W} : \hn \otimes W \to W \otimes \hn,
  \]
and the braiding map
  \[
    \Phi_{\brh, W}   : \brh \otimes W \to W \otimes \brh
  \]
  comes from $\Phi_{\brnh,W}$ applied degree by degree.
  Similarly for the braiding maps
  $\Phi_{W, \brnh}$ and $\Phi_{W, \brh}$.
  \end{defi}

Note that $\brh$ need not be a subalgebra of
$\Th$. Instead, we define a different multiplication on it. 
To this end, we generalize the notion of symmetrization in 
Definition~\ref{defi:symmetrization} in the covariant setting.

\begin{defi}
  A \emph{braidization} in $\bc$ is the following data.
  \begin{itemize}
    \item (\emph{$n$th braidization}) For each $H \in \bcf$ and each $n \geq 0$, 
  a surjective linear map 
  \[
    \mathpzc{B}_n : 
  H^{\otimes n} \to \mathpzc{Br}^n H
  \] 
  that satisfies
  \begin{enumerate}[(i)]
    \item $\mathpzc{B}_n ^2 = \mathpzc{B}_n \ ,$ 
    \item $\mathpzc{B}_n (v) = \mathpzc{B}_n (b \cdot v)$ for any 
      $v \in H^{\otimes n}$ and $b \in B_n \ .$ 
  \end{enumerate}
\end{itemize}
These maps should satisfy the following two conditions.
\begin{itemize}
\item (\emph{Functoriality}) For each morphism $\phi : H \to K$ in $\bcf$,
  the following diagram is commutative for each $n$.
  \[
    \begin{CD}
\hn @>\bn>> \brnh \\
@VV\pn V      @VV\pn V \\
\kn @>\bn>> \brnk
    \end{CD}
  \]

\item (\emph{Associativity}) For each $H \in \bcf$,  $v \in \hn$, $w \in \hm$ and
  $z \in \hk$,
\begin{align*}
    \B_{n+m+k}(\B_{n+m}(v \otimes w)\otimes z) 
    &= \B_{n+m+k}(v \otimes \B_{m+k}(w \otimes z)).
  \end{align*}
  \end{itemize}
  $\mathbf{C}$ is \emph{braidizable}
if $\bc$ admits a braidization.
After fixing a braidization, $\bc$ is \emph{braidized}.
If $\bc$ is braidized, for each $H \in \bcf$, set the \emph{total braidization}
  \[
    \mathpzc{B}:= \prod_{n=0}^{\infty} \mathpzc{B}_n : 
\Th \to \brh \, .
  \]
 \label{defi:braidization}
\end{defi}

\begin{rem}
  Note that the surjectivity and condition (i) of $\mathpzc{B}_n$
  implies that it is the identity on $\mathpzc{Br}^n H$.
  \label{rem:identity}
\end{rem}

Now we are ready to generalize Proposition~\ref{pro:symmetrization} 
in the covariant setting.

\begin{pro}
  Suppose that $\bc$ is regular and braidized.
    Then for each $H \in \bcf$, the composition of the following maps
  \begin{align*}
    \circ:
    \mathpzc{Br} \llbracket H \rrbracket \otimes \mathpzc{Br} \llbracket H \rrbracket
    \hookrightarrow \Th \otimes 
    \Th \xrightarrow{\mu} 
    \Th \xrightarrow{\mathpzc{B}}
    \mathpzc{Br} \llbracket H \rrbracket
  \end{align*}
  gives $\mathpzc{Br} \llbracket H \rrbracket$ the structure of a 
  braided-commutative algebra with unity.
  Moreover, for any
  morphism $\phi: H \to K$ in $\bcf$,
  the induced linear map
  \[
\phi_\ast : \brh \to \brk
  \]
  is a morphism of algebra with unity.
  \label{pro:braidedalgebra}
\end{pro}

\begin{proof}
Braided-commutativity follows from the following equations.
Let $\sum_{m=0}^{\infty} v_m$, $\sum_{n=0}^{\infty} w_n \in \brh$ with
$v_m \in \mathpzc{Br}^m H$ and
$w_n \in \brnh$.
Using the regularity of $\bc$, we have
\begin{align*}
  \mathpzc{B}(\mu(&\sum_{m=0}^{\infty} v_m \otimes \sum_{n=0}^{\infty} w_n)) 
  = \mathpzc{B}(\sum_{m,n=0}^{\infty}v_m \otimes w_n) \\ 
  &= \mathpzc{B} (\sum_{m,n=0}^{\infty} \Phi(v_m \otimes w_n)) \quad 
  \text{by (ii) of Definition~\ref{defi:braidization}}\, , \\
  &= \mathpzc{B} (\mu \Phi(\sum_{m=0}^{\infty} v_m \otimes \sum_{n=0}^{\infty} w_n))\, . 
\end{align*}
Associativity of $\circ$ follows from that of the braidization.
Unity exists by Remark~\ref{rem:identity}.

We have the following commutative diagram by the functoriality
 of the braidization in $\bc$.
\[
  \begin{CD}
\brnh \otimes \brmh @>\mu>> H^{\otimes (n+m)} @>\B_{n+m}>> \mathpzc{Br}^{n+m}H \\
@VV\phi^{\otimes (n+m)}V     @VV\phi^{\otimes (n+m)}V       @VV\phi^{\otimes (n+m)}V \\
\brnk \otimes \brmk @>\mu>> K^{\otimes (n+m)} @>\B_{n+m}>> \mathpzc{Br}^{n+m}K 
  \end{CD}
\]
Recall that $\mu$ in the left box denotes the juxtaposition.
This diagram shows that $\phi_\ast$ preserves the multiplication.
Unity is preserved since $\phi^{\otimes n}  = id_k$ on $\brnh = \brnk = k$
when $n = 0$.
\end{proof}

\subsection{A Groupoid Structure on $G^n$}
\label{ssec:groupoidgn}

Before specializing to the case of $G$-graded $G$-modules,
we study the $B_n$-action on the $G^n$-degrees on $\hn$ 
for any $G$-graded $G$-module $H$.

$G^n \rtimes S_n$ acts on $G^n$ by 
taking the componentwise adjoint action of $G^n$ first, 
and then switching positions. 
This induces a natural $B_n$ action on $G^n$ in the following way.
Let $b_1$ be one of the standard generators of $B_n$ that braids
the first two strands. Then
\begin{align*}
  b_1 (\gamma_1, \dots, \gamma_n) := 
  (\gamma_1 \gamma_2 \gamma_1^{-1}, \gamma_1, \gamma_3, \dots, \gamma_n) \, ,
\end{align*}
and similarly for the other generators.
This shows that given any $b \in B_n$ and $\boldsymbol \gamma \in G^n$,
we have an element $b_{\boldsymbol \gamma} \in G^n \rtimes S_n$
that really acts on $\boldsymbol \gamma$.
For example, we see that $b_{1, \bg} = (e, \gamma_1, e, \dots, e) \times (1,2)$.
Note that this does not define a homomorphism from $B_n$ to $G^n \rtimes S_n$
since $b_{\bg}$ depends on $\bg$.
We have the following structure instead.

\begin{defi}
  The \emph{$B_n$-groupoid $\mathcal{G}^n$} is defined in the following way:
\begin{align*}
   \obj(\mathcal{G}^n) &= G^n \, , \\
   \hom (\boldsymbol \gamma_1, \boldsymbol \gamma_2) &= 
    \{b_{\boldsymbol \gamma_1} \in G^n \rtimes S_n | b \boldsymbol \gamma_1 =
    \boldsymbol \gamma_2 \text{ for some } b \in B_n \} \, .
  \end{align*}

  \begin{rem}
    A related structure is the \emph{action groupoid} $B_n \ltimes G^n$.
    See \cite{moer} for its definition.
    This is different from $\gn$ in that different elements of $B_n$
    can give rise to the same arrow in $\gn$.
    In particular, $\gn$ is a finite groupoid, whereas $B_n \ltimes G^n$ is not.
  \end{rem}

  \begin{rem}
 $\obj(\gn)$ has the natural $G$-grading that is defined as 
 $\prod \gamma_i$ for $\bg \in \obj(\gn)$.
Since it is invariant under the action of $B_n$, 
each component of $\mathcal{G}^n$ has the constant
    $G$-degree.
    \label{rem:degree}
  \end{rem}

  \begin{rem}
    Consider the diagonal adjoint $G$-action on $\obj(\gn)$.
We observe that this commutes with the action of $B_n$. 
This implies that $G$ acts on the set of components 
of $\gn$. If a component $C$ has the $G$-degree $\gamma$,
the $G$-degree of $g \cdot C$ is equal to $g \gamma g^{-1}$.
    \label{rem:diagonal}
  \end{rem}

\end{defi}

We want to introduce the concept of reflection map on $B_n$, that will be useful
when we consider braided tensors. 

\begin{defi}
  The \emph{reflection} $r$ on $B_n$ is an \emph{anti}homomorphism that sends
  $b_i$ to $b_{n-i}$.
  \label{defi:breflection}
\end{defi}

\begin{rem}
  This definition makes sense since the defining relations on the generators of
  $B_n$ are also satisfied among the images of them.
  $r$ being an antihomomorphism does not matter because of the form of the relations.
\end{rem}

We want to see how this map works in the categories $\gn$.
Let $\bg :=(\gamma_1, \dots, \gamma_n) \in \gn$.
Let 
\[
  r: G^n \to G^n \, , \quad (\gamma_1, \dots, \gamma_n) 
  \mapsto (\gamma_n^{-1} , \dots, \gamma_1^{-1})
\]
be a set map.
Then we have the following commutative diagram
\[
  \begin{CD}
(\gamma_1, \dots, \gamma_n) @>r>> (\gamma_n^{-1}, \dots, \gamma_1^{-1}) \\
@VVb_1V                              @AAb_{n-1}A \\
(\gamma_1 \gamma_2 \gamma_1^{-1}, \gamma_1, \gamma_3, \dots, \gamma_n) @>r>>
(\gamma_n^{-1}, \dots, \gamma_3^{-1}, \gamma_1^{-1}, \gamma_1 \gamma_2^{-1} \gamma_1^{-1})
  \end{CD}
\]
and similarly for the other generators of $B_n$.
This allows us to make the following definition.

\begin{defi}
  The \emph{reflection map} on $\gn$ is the functor $r: (\gn)^\circ \to \gn$ that
  maps $\bg$ to $r(\bg)$ in the sense above, and sends $b_{i,\bg}$ to 
  $b_{n-i, r(b_i \cdot \bg)}$.
  \label{defi:reflection}
\end{defi}

\begin{rem}
  Given a morphism in $\gn$, 
  note that we should choose an element in $B_n$ that acts as the given one 
  to know the image of it under $r$. 
  One sees that this does not depend on the choice of the representative by looking
  at the diagram above.
  Namely, suppose we return to the original $\bg$ after several application
  of $b_i$'s and its inverses on the left side of the above diagram.
  Then the same thing should happen to the right side since the right
  side is the exact ``reflection'' of the left side.

  This observation also implies the following proposition.
\end{rem}

\begin{pro}
  $r$ is fully faithful.
  \label{pro:fullyfaithful}
\end{pro}

\subsection{$G$-braided Spaces}
\label{ssec:gbrsp}

In this section, we consider the category of $G$-graded
$G$-modules, and interpret its finite dimensional objects as
 braided-commutative spaces.

\begin{defi}
  A \emph{$G$-graded $G$-module} $(H,\rho)$ is a vector space $H$
  with the decomposition 
  \begin{equation*}
    H=\bigoplus_{m \in G}H_m \ 
    \label{eq:grading}
  \end{equation*}
and a $G$-action $\rho$ respecting the $G$-grading, in the sense that
\begin{equation*}
  \rho(\gamma)H_m = H_{\gamma m\gamma^{-1}}\ 
  \label{eq:action}
\end{equation*}
for any $\gamma, m \in G$.
\label{defi:ggradedgmodule}
\end{defi}

\begin{rem}
  We may suppress $\rho$ for the $G$-action and use $\cdot$ instead,
  when no confusion should arise.
\end{rem}

\begin{rem}
  $G$-graded $G$-modules form the category
  $\dcg$-$\md$,
  the category of $D(k[G])$-modules \cite{dt, majid}. The latter is the Drinfel'd double
  of the group algebra $k[G]$. It is a quasitriangular Hopf algebra, 
  hence $\dcg$-$\md$ is a braided monoidal category. The braiding 
  is given by 
  \[
  \Phi : V_g \otimes W_h \to W_{ghg^{-1}} \otimes V_{g} \, , \quad
  w \otimes v \mapsto g \cdot v \otimes w \, 
\]
  and extending linearly.
  Their morphisms are $G$-equivariant linear maps that preserve the $G$-gradings.
  See \cite{majid} for details.
  \label{rem:category}
\end{rem}

\begin{rem}
  A (linear) representation of a groupoid is a functor 
  from the groupoid to the category of $k$-vector spaces.
  Given a $G$-graded $G$-module $H$, $H^{\otimes n}$ gives rise to
  a $\mathcal{G}^n$-representation.
  For $\boldsymbol \gamma :=(\gamma_1, \dots, \gamma_n)$,
  let $H_{\boldsymbol \gamma} := H_{\gamma_1} \otimes \dots \otimes H_{\gamma_n}$.
  Then we have $H^{\otimes n} = \bigoplus H_{\boldsymbol \gamma}$.
  The representation assigns $\boldsymbol \gamma$ to $H_{\boldsymbol \gamma}$
  and $b_{\boldsymbol \gamma}$ to $b:H_{\boldsymbol \gamma} \to H_{b \boldsymbol \gamma}$.
  This is well defined since if $b_{\boldsymbol \gamma} = b'_{\boldsymbol \gamma}$,
  the maps $b$ and $b' : H_{\boldsymbol \gamma} \to H_{b \boldsymbol \gamma} = 
  H_{b' \boldsymbol \gamma}$ are the same.
    \label{rem:representation}
\end{rem}

\begin{rem}
When $G$ is trivial, $\dcg$-$\md$ is nothing but the symmetric monoidal 
category of vector spaces.
Hence its braided multilinear forms on finite dimensional objects
are the same as the symmetric multilinear forms.
\label{rem:trivial}
\end{rem}

  Note that $\brnh$ is graded by components of $\gn$.
That is, let $C$ be a component of $\gn$.
Then $H_C := \bigoplus_{\bg \in C} H_{\bg}$ is an invariant subspace of the map
\[
b : \hn \to \hn \, ,
\]
for any $b \in B_n$.
Let $C(\gn)$ denote the set of components in $\gn$ and
define $\brnch := \brnh \cap H_C$.
Then we have
\[
  \brnh = \bigoplus_{C \in C(\gn)} \brnch \, ,
\]
where $B_n$ acts on $\brnch$ trivially.
This follows from the equation $b \cdot v = v$ for any $v \in \brnh$ and $b \in B_n$.
We give a characterization of vectors in $\brnch$. 

\begin{pro}
  Let $C \in C(\gn)$ and $v \in H_C$. Write 
\[ 
  v = \sum_{\bg \in C} v_{\bg} \, , 
\]
where $v_{\bg}$ denotes the $G^n$-homogeneous component of $v$.
Then $v \in \brnch$ if and only if 
\[
  b \cdot v_{\bg} = v_{b \cdot \bg}
\]
for any $b \in B_n$.
\label{pro:brnchs}
\end{pro}

\begin{proof}
 Suppose $v \in \brnch$.
 Then the equation $b \cdot v = v$ implies
 \[
   \sum_{\bg \in C} b \cdot v_{\bg} = \sum_{\bg \in C} v_{\bg} \, .
 \]
 Comparing the homogeneous components yields the desired equation.
 Conversely, suppose $b\cdot v_{\bg}= v_{b \cdot \bg}$ for $b \in B_n$.
 We compute
 \[
   b \cdot v = \sum_{\bg \in C} b \cdot v_{\bg} 
   = \sum_{\bg \in C} v_{b \cdot \bg} = \sum_{\bg \in C} v_{\bg} = v \, ,
 \]
since the map $b: C \to C$, defined by acting on elements in $C$,
is bijective.
\end{proof}

Now we are ready to prove that $\dcg$-$\md$ is regular and braidizable.

\begin{pro}
  $\dcg$-$\md$ is regular.
  \label{pro:regular}
\end{pro}

\begin{proof}
 Let $H$ be a finite dimensional $G$-graded $G$-module. 
 To see the $G$-graded $G$-module structure of $\brnh$, 
note that $\brnch$ has the natural $G$-degree equal to that of $C$
as defined in Remark~\ref{rem:degree}.
This defines a $G$-grading on $\brnh$.

To prove the $G$-module structure, consider the  diagonal $G$-action
on $\hn$. We want to show that $\brnh$ is invariant under this action.
Suppose that 
\[
v = \sum_{\bg \in C} v_{\bg} \in \brnch
\]
for some $C \in C(\gn)$.
Then we have
\[
  w := g \cdot v = \sum_{\bg \in C} g \cdot v_{\bg} \, ,
\]
and set $w_{g \cdot \bg} :=g \cdot v_{\bg}$, so that
\[
  w = \sum_{\bg \in g \cdot C} w_{\bg} \, .
\]
We use Proposition~\ref{pro:brnchs} to show that $w \in \brngch$.
Recall from Remark~\ref{rem:diagonal} that the $B_n$ and $G$-actions on
$\obj(\gn)$ commutes with each other. 
Moreover, we observe that the two actions on $\hn$ also commute with each other.
Hence for any $b \in B_n$,
\begin{align*}
   b \cdot w_{\bg} = b &\cdot (g \cdot v_{g^{-1} \cdot \bg}) 
   = g \cdot (b \cdot v_{g^{-1} \cdot \bg}) \\
   &= g \cdot v_{b \cdot (g^{-1} \cdot \bg)} 
   = g \cdot v_{g^{-1} \cdot(b \cdot \bg)} = w_{b \cdot \bg} \, .
\end{align*}
The compatibility of its $G$-grading and $G$-action follows by observing that
\[
\deg_G (g \cdot C) = g \cdot \deg_G C \cdot g^{-1} 
\]
where $\deg_G$ denotes the $G$-degree.

Next, we prove that $\brh$ is a $G$-graded $G$-module.
Using the $G$-grading on $\brnh$, write
\[
  \brnh = \bigoplus_{g \in G} (\brnh)_g
\]
where the subscript denotes the $G$-degree.
Using the finiteness of $G$, we have
\begin{align*}
  \brh &= \prod_{n=0}^{\infty}\brnh 
  = \prod_{n=0}^{\infty} \bigoplus_{g \in G} (\brnh)_g \\
  &= \bigoplus_{g \in G} \prod_{n=0}^{\infty} (\brnh)_g
  = \bigoplus_{g \in G} (\brh)_g \, ,
\end{align*}
by setting that $(\brh)_g := \prod_{n=0}^{\infty} (\brnh)_g$.
$G$-action on $\brh$ is defined term by term diagonally. 
The  $G$-invariance of $\brnh$ guarantees that of $\brh$.
The compatibility of its $G$-grading and $G$-action also follows from that of $\brnh$.

The naturality of their braidings follows from the fact that the $G$-grading 
on $\brnh$ is inherited from $\hn$.
\end{proof}

\begin{pro}
  $\dcg$-$\md$ is braidizable.
  \label{pro:braidization}
\end{pro}

\begin{proof}
  (\emph{$n$th braidization}) Let $H$ be a finite dimensional $G$-graded $G$-module.
  We use the $\mathcal{G}^n$-representation structure of Remark~\ref{rem:representation}
  on $\hn$. For $\boldsymbol \gamma \in \obj (\mathcal{G}^n)$, let
  $A_{\boldsymbol \gamma}$ 
  be the set of all the morphisms in the component of $\boldsymbol \gamma$ 
  with the source $\boldsymbol \gamma$. That is,

  \begin{align*}
    A_{\boldsymbol \gamma}:=\coprod_{\substack{\boldsymbol \gamma' = b \boldsymbol \gamma \\
    \text{for some}\ b \in B_n}} \hom_{\mathcal{G}^n}
    (\boldsymbol \gamma , \boldsymbol \gamma') \, .
  \end{align*}

  For $v \in H_{\bs \gamma}$, define

  \begin{align*}
    \mathpzc{B}_n(v) 
    := \frac{1}{|A_{\boldsymbol \gamma}|} 
    \sum_{b_{\boldsymbol \gamma} \in A_{\boldsymbol \gamma}} b_{\boldsymbol \gamma}
    \cdot v \, ,
  \end{align*}
  and extend linearly.
  
  We show that
\[
  b_i \cdot \bn (v) = \bn (v)
\]
for any generator $b_i$ of $B_n$. It is still enough to assume that $v \in H_{\bg}$.
Write
\[
  b_i \cdot \bn (v) = \frac{1}{|A_{\bg}|} \sum_{b_{\bg} \in A_{\bg}}
  b_i \cdot (b_{\bg} \cdot v)
\]
and observe that the map
\[
  b_i : A_{\bg} \to A_{\bg}
\]
defined by postcomposing the element $b_{i, \bg'} \in G^n \rtimes S_n$ 
is a bijection.
Here $\bg'$ denotes the target of each arrows in $A_{\bg}$.
Hence 
\[
  b_i \cdot \bn (v) = 
  \frac{1}{|A_{\bg}|} \sum_{b'_{\bg} \in A_{\bg}}
  b'_{\bg} \cdot v = \bn (v) \, ,
\]
where $b'_{\bg}$ denotes each element in the image of the map $b_i$ above.

To show the surjectivity and condition (i) of Definition~\ref{defi:braidization},
it is enough to show that 
\[
\mathpzc{B}_n (v) = v
\]
for any $v \in \brnh$.
Considering the $C(\gn)$-grading of $\brnh$,
it is enough to consider that $v \in \brnch$ for $C \in C(\gn)$ and put
\[
  v = \sum_{\bg \in C} v_{\bg} \, .
\]

Observe that $|A_{\bg}|$ is constant for $\bg \in C$. Write
\[
  n_C := |A_{\bg}|
\]
for any choice of $\bg \in C$. Also, $|\hom_{\gn}(\bg , \bg')|$ is
constant for any two objects $\bg$ and $\bg'$ in $C$.
Define
\[
m_C := |\hom_{\gn}(\bg , \bg')|
\]
for any choice of two objects $\bg$ and $\bg'$ in $C$. Then we have
\begin{align}
  n_C = |C| \cdot m_C \, .
  \label{eq:nummorphism}
\end{align}
By the definition of $\bn$, we have
\[
  \bn (v) = \frac{1}{n_C} \sum_{\substack{\bg \in C \\
    b_{\bg} \in A_{\bg}}} b_{\bg} \cdot v_{\bg} \, .
\]
Proposition~\ref{pro:brnchs} implies
\[
  \sum_{b_{\bg} \in A_{\bg}} b_{\bg} \cdot v_{\bg} = m_C \cdot v 
\]
for any fixed $\bg \in C$. 
Thus
\begin{align*}
   \bn (v) = \frac{1}{n_C} \sum_{\bg \in C} m_C \cdot v 
  = \frac{1}{n_C} \cdot |C| \cdot m_C \cdot v 
  = v \, ,
\end{align*}
by Equation~\eqref{eq:nummorphism}.

To prove condition (ii), it is enough to consider $v \in H_{\bg}$. 
Then $b \cdot v \in H_{b \cdot \bg}$. Hence
\[
  \bn (b \cdot v) = \frac{1}{|A_{b \cdot \bg}|}
  \sum_{b_{b\cdot \bg} \in A_{b \cdot \bg}} 
  b_{b \cdot \bg} \cdot (b \cdot v) \, .
\]
Since the map
\[
  b : A_{b \cdot \bg} \to A_{\bg}
\]
defined by precomposing $b_{\bg}$ is a bijection, and using the fact that
$|A_{b \cdot \bg}| = |A_{\bg}|$, 
we have
\[
  \bn(b \cdot v) = \frac{1}{|A_{\bg}|} \sum_{b'_{\bg} \in A_{\bg}}
  b'_{\bg} \cdot v \, ,
\]
where $b'_{\bg}$ denotes each element in the image of the map $b$ above.
  
(\emph{Functoriality}) It is enough to consider $G^n$-homogeneous vectors.
  Hence, suppose $v \in H_{\bg}$. Then by the
  definition of $n$th braidization, we have
  \[
    \pn(\bn(v)) = \frac{1}{|A_{\bg}|} 
    \sum_{b_{\bg} \in A_{\bg}} \pn(b_{\bg} \cdot v) \, .
  \]
  For each $b_{\bg} \in A_{\bg}$, choose $b \in B_n$ such that
  $b$ acts as $b_{\bg}$ on $H_{\bg}$. Let $A'$ be the finite subset of $B_n$
  of such $b$'s.
  Then we have 
  \begin{align*}
    \pn (\bn (v)) = \frac{1}{|A_{\bg}|} \sum_{b \in A'}
  \pn(b \cdot v) 
  = \frac{1}{|A_{\bg}|} \sum_{b \in A'} b \cdot \pn(v)
  \end{align*}
  by Lemma~\ref{lem:commute}.
  Since $\pn$ preserves the $G^n$-grading, we have
  \[
    \pn (\bn (v)) = \frac{1}{|A_{\bg}|} \sum_{b_{\bg} \in A_{\bg}}
    b_{\bg} \cdot \pn (v) = \bn(\pn (v)) \, .
  \]

(\emph{Associativity}) Suppose $v$, $w$ and $z$ in $\Th$
  are homogeneous of degree $n$, $m$ and $k$, respectively.
  By linearity, it is enough to assume that they are also monomials of
  homogeneous $G^n$, $G^m$ and $G^k$-degrees, respectively.
  We observe that
  \begin{align*}
    \B_{n+m+k}(\B_{n+m}(v \otimes w)\otimes z) 
    &= \B_{n+m+k}(v \otimes w \otimes z) \\
    &= \B_{n+m+k}(v \otimes \B_{m+k}(w \otimes z))
  \end{align*}
  since the $\g^{n+m+k}$-orbits of monomials of
  $\B_{n+m}(v \otimes w)\otimes z$
  and
$v \otimes \B_{m+k}(w \otimes z)$ coincide with that of
$v \otimes w \otimes z$, with correct averaging constants.
\end{proof}

\begin{rem}
  We assume that $\dcg$-$\md$ is braidized in this way hereafter.
\end{rem}

\begin{rem}
  Note that we recover the usual symmetrization when $G$ is trivial.
\end{rem}

The following proposition enables us to use the results of the previous
section in the dual setting.

\begin{pro}
  In $\dcg$-$\md$, taking the dual space is a contravariant functor on its
  full subcategory of finite dimensional $G$-graded $G$-modules.
    \label{pro:dual}
\end{pro}

\begin{proof}
Given a finite dimensional $G$-graded $G$-module $H$,
its induced $G$-action on $H^\ast$ is given by the following formula.
\begin{equation}
  (\gamma \cdot x) (v) = x (\gamma ^{-1} \cdot v) \, .
  \label{eq:induced}
\end{equation}
One consistent choice of $G$-grading on $H^\ast$ can be obtained by requiring
that the trace map
\[
\tr: H \otimes H^\ast \to \C \, , \quad v \otimes x \mapsto x(v)
\]
are $G$-degree preserving.
Namely, the $G$-grading on $H^\ast$ can be given as
 \[ 
  (H^{\ast})_m = (H_{m^{-1}})^{\ast}\, .
\]
We observe that Equation~\eqref{eq:induced} implies that
\[
  \gamma \cdot (H^\ast)_{m} = (H^\ast)_{\gamma m \gamma^{-1}} \, .
\]

Let $\phi : H \to K$ be a morphism between finite dimensional objects.
Then $\phi^\ast : K^\ast \to H^\ast$ is a linear map.
To prove that $\phi^\ast$ preserves the $G$-degree, suppose that 
$x \in (K^\ast)_m$, $x \neq 0$.
If $\phi^\ast x = 0$ for any $x$, then $\phi^\ast$ preserves degree $m$.
If not, there exist a homogeneous vector $v \in H$ such that
$\phi^\ast x (v) \neq 0$.
But $\phi^\ast x (v) = x (\phi(v))$ and $\phi$ is degree-preserving,
it follows that $\deg_G v = m^{-1}$.
We conclude that $\phi^\ast x$ is homogeneous of $G$-degree $m$.
To show the $G$-equivariance, let $x \in K^\ast$ and $v \in H$. 
Using the $G$-equivariance of $\phi$, we compute
\begin{align*}
\phi^\ast(g \cdot x)(v) &= (g \cdot x)(\phi(v)) = x(g^{-1} \cdot \phi(v)) \\
&= x(\phi(g^{-1} \cdot v)) = \phi^\ast x(g^{-1} \cdot v) = g \cdot (\phi^\ast x)(v) \, .
\end{align*}

The conditions on composition and identity follows from that of linear maps.
\end{proof}

\begin{rem}
  We fix this $G$-graded $G$-module structure for dual spaces hereafter.
\end{rem}

The following theorem summarizes what we have been proving.

\begin{thm}
  Let $H$ be a finite dimensional $G$-graded $G$-module, and 
  $\circ$ be defined as in Proposition~\ref{pro:braidedalgebra}.
  Then $(\brhs,\circ)$ is
  a braided-commutative algebra with unity.
  Moreover, for each morphism $\phi : H \to K$ of finite dimensional 
  $G$-graded $G$-modules, the induced linear map 
  \[
\phi^\ast : \brks \to \brhs
  \]
  is a morphism of algebra with unity.
  \label{thm:ringmap}
\end{thm}

\begin{proof}

  This follows from Propositions~\ref{pro:braidedalgebra},  
  \ref{pro:regular}, \ref{pro:braidization}, and \ref{pro:dual}.

\end{proof}

\begin{rem}
  We call the elements of $\brnhs$ and $\brhs$ the \emph{braided $n$-linear forms}
  and \emph{braided multilinear forms}, respectively.
  \label{rem:braidedmultilinear}
\end{rem}

\begin{rem}
  We adopt the viewpoint that any finite dimensional $G$-graded $G$-module is
  a braided-commutative space with its ring of braided
  multilinear forms.
  \label{rem:trivialg}
\end{rem}

Lastly, we compare the $B_n$ actions on $\hn$ and $\hsn$.
We use the notation of Definition~\ref{defi:breflection} for the following 
proposition.

\begin{pro}
  For any $x \in \hsn$ and $v \in \hn$, the following equations hold.
  \[
(b \cdot x) (v) = x (r(b) \cdot v) \, .
  \]
  \label{pro:starbraid}
\end{pro}

\begin{proof}
 Recall that the trace map is given by
 \begin{align*}
 \tr: \hsn \otimes \hn &\to \C \, ,\\ 
 x_1 \otimes \dots \otimes x_n 
\otimes v_1 \otimes \dots \otimes v_n 
&\mapsto x_n(v_1)\cdot x_{n-1}(v_2) \cdots x_1(v_n)
 \end{align*}
 so that there is no need of braiding.

 To prove the equation, it is enough to 
prove the following equations
  \begin{align*}
    (b_i \cdot x) (v) &= x (b_{n-i} \cdot v) \, , \\
    (b_i^{-1} \cdot x) (v) &= x (b_{n-i}^{-1} \cdot v) \, .
  \end{align*}

 By linearity, we assume that $x$ is
 a $G^n$-homogeneous monomial.
 For simplicity, we only consider $i=1$.
 Suppose that $x \in H^\ast_{\bs \gamma}$ in the notation of 
 Remark~\ref{rem:representation}, and write
 \[
   x = x_{\gamma_1} \otimes \dots \otimes x_{\gamma_n} \, .
 \]
 Then we have
 \[
   b_1 \cdot x = \gamma_1 \cdot x_{\gamma_2} \otimes
   x_{\gamma_1} \otimes x_{\gamma_3} \otimes \dots \otimes
   x_{\gamma_n} \, .
 \]
 It is enough to consider $G^n$-homogeneous $v$ of matching $G^n$-degrees, otherwise
 both sides are equal to zero. Hence, write
 \[
   v = v_{\gamma_n^{-1}} \otimes \dots \otimes
   v_{\gamma_3^{-1}} \otimes v_{\gamma_1^{-1}} \otimes v_{\gamma_1 \gamma_2^{-1} \gamma_1^{-1}} \, .
 \]
 Then we have
 \[
   b_{n-1} \cdot v = v_{\gamma_n^{-1}} \otimes \dots \otimes
   v_{\gamma_3^{-1}} \otimes \gamma_1^{-1} \cdot v_{\gamma_1 \gamma_2^{-1} \gamma_1^{-1}} \otimes
   v_{\gamma_1^{-1}} \, .
 \]
 We compute
 \begin{align*}
   (b_1 \cdot x)(v) 
   &= x_{\gamma_n}(v_{\gamma_n^{-1}}) \cdots
   x_{\gamma_3}(v_{\gamma_3^{-1}}) \cdot x_{\gamma_1}(v_{\gamma_1^{-1}}) \cdot 
   (\gamma_1 \cdot x_{\gamma_2})(v_{\gamma_1 \gamma_2^{-1} \gamma_1^{-1}}) \\
   &= x_{\gamma_n}(v_{\gamma_n^{-1}}) \cdots
   x_{\gamma_3}(v_{\gamma_3^{-1}}) \cdot x_{\gamma_1}(v_{\gamma_1^{-1}}) \cdot 
   x_{\gamma_2}(\gamma_1^{-1} \cdot v_{\gamma_1 \gamma_2^{-1} \gamma_1^{-1}}) \\
   &= x(b_{n-1} \cdot v) \, 
 \end{align*}
 by Equation~\eqref{eq:induced}.
 The second equation can be proved in a similar fashion.
\end{proof}

This leads to the following characterization of braided multilinear forms
on $H$.

\begin{cor}
  For any finite dimensional $G$-graded $G$-module $H$, we have
  \[
    \brnhs =  \{ x \in \hsn | x(v) = x(b\cdot v) \
      \text{for any} \ b \in B_n \ \text{and} \ v \in \hn \}\, ,
  \]
  the subspace of $B_n$-invariant multilinear forms. 
  \label{cor:braided}
\end{cor}

\begin{proof}
\begin{flalign*}
  &x \in \brnhs &\\
\iff &x \in [\hsn]^{B_n}  \\
\iff &x = b \cdot x \, , \quad \forall \, b \in B_n \, \\
\iff &x(v) = (b_{i} \cdot x)(v) = x(b_{n-i} \cdot v) \, , 
\quad \{b_i\} \ \text{standard generators of}\ B_n \, , \\
&\forall \, v \in \hn \\
\iff &x(v) = x(b \cdot v) \, ,\quad  \forall \, b \in B_n \, ,\quad \forall \, v \in \hn \, .
\end{flalign*}
\end{proof}

The previous proposition also allows us to relate the $n$th braidization 
on $\brnh$ and $\brnhs$. 

\begin{cor}
For any $x \in \hsn$ and $v \in \hn$, we have
\[
  [\bn(x)](v) = x(\bn(v)) \, .
\]
\label{cor:relatedbraidization}
\end{cor}

\begin{proof}
  We use the notation of Definition~\ref{defi:reflection}.  
  
  Let $x \in H^\ast_{\bg}$. 
  Suppose that $C$ is the component of $\bg \in \gn$ and choose $\bg' \in C$.
  We assume that $v \in H_{r(\bg')}$.
  Otherwise both sides are zero.
  Then we have
\[
  [\bn(x)](v) = \frac{1}{|A_{\bg}|} 
  \left[\sum_{b_{\bg} \in A_{\bg}} b_{\bg} \cdot x \right] (v)
  = \frac{1}{|A_{\bg}|} \sum_{b_{\bg} \in \hom(\bg, \bg')} 
  \left[(b_{\bg} \cdot x)(v) \right] \, .
\]
Choose $b \in B_n$ for each $b_{\bg} \in \hom_{\gn}(\bg, \bg')$
so that $b$ acts as $b_{\bg}$ on $H^\ast_{\bg}$.
Let $A'$ be the subset of $B_n$ of such $b$'s. 
Using Proposition~\ref{pro:starbraid} we have
\begin{align*}
  [\bn(x)](v) &= \frac{1}{|A_{\bg}|} \sum_{b \in A'} \left[ (b \cdot x) (v) \right] \\
  &= \frac{1}{|A_{\bg}|} \sum_{b \in A'} x(r(b) \cdot v) 
  = x \left( \frac{1}{|A_{\bg}|} \sum_{b \in A'} r(b) \cdot v \right)\, .
\end{align*}
By Proposition~\ref{pro:fullyfaithful}, we see that the action of 
$\left\{ r(b)| b \in A' \right\}$
is the same as that of $\hom_{\gn}(r(\bg'), r(\bg))$. 
Since $x \in H^\ast_{\bg}$ and 
$|A_{\bg}| = |A_{r(\bg')}|$ by Proposition~\ref{pro:fullyfaithful},
we conclude
\[
  [\bn(x)](v) = x (\bn (v)) \, .
\]

\end{proof}

\subsection{Distinguished Subspaces}

For any finite dimensional $G$-graded $G$-module $H$, we have two distinguished
subspaces $H_e$ and $H^G$. $H_e$ denotes the $G$-homogeneous
subspace of $G$-degree $e$.
$H^G$ is the subspace of $G$-invariants.

Let us first consider $H_e$. 
Since $H_e$ is a $G$-graded $G$-module concentrated at $G$-degree $e$, 
we can talk about its braided multilinear forms.
But in a similar way as Remark~\ref{rem:trivial},
notice that the braided multilinear forms are the same as
symmetric multilinear forms in this case.
Remark~\ref{rem:trivialg} shows that this identification is as rings.
Let $i_e : H_e \hookrightarrow H$ be the inclusion map.
Note that this is a morphism of $G$-graded $G$-modules.
Applying Theorem~\ref{thm:ringmap}, we have the following proposition.

\begin{pro}
 We have the following induced morphism of algebras with unity.
 \[
(i_e)^\ast : \brhs \to \br \llbracket H_e^\ast \rrbracket 
= \sym \llbracket H^\ast_e \rrbracket
 \]
 \label{pro:untwisted}
\end{pro}

Next, we consider the subspace of $G$-invariants, $H^G$.
Let $\bar{g} \in \bar{G}$, the conjugacy class of $g$.
Since
$\bigoplus_{h \in \overline{g}} H_h $ is an invariant space for any $g \in G$, 
we see that $H^G$ has a $\overline{G}$-grading.
Let $H^G_{\overline{g}}$ denote the homogeneous part of $\bar{G}$-degree $\bar{g}$.

In particular, $H^G$ is not a $G$-graded $G$-module since
$H_h$ is not an invariant subspace of 
the map $g : H \to H$ in general.
Hence we cannot talk about braided multilinear forms on $H^G$, but
we can consider the restrictions of those of $H$ on $H^G$.

\begin{pro}
 Let $i^G : H^G \hookrightarrow H$ be the inclusion.
 Then we have the linear map
\[
  (i^G)^\ast : \brhs \to \shgs \, , \quad \sum_{n=0}^{\infty} x_n 
  \mapsto \sum_{n=0}^{\infty} x_n|_{\hgn} \, ,
\]
where $x_n \in \brnhs$.
\label{pro:invarsym}
\end{pro}

\begin{proof}
 Let $x \in \brnhs$ and $v \in \hgn \subset \hn$. 
 Let $s : B_n \to S_n \, , b_i \mapsto (i, i+1)$ be the standard surjective homomorphism.
 Then for any $b \in B_n$, we have $b \cdot v = s(b) \cdot v$.
 To see this, suppose $v = v_1 \otimes \dots \otimes v_n$ with $v_i \in H^G$.
 Suppose further that $v_1 = \sum_{g \in G} v_g$, where $v_g \in H_g$.
 Note that each $v_g$ need not be in $H^G$.
 We consider $b_1 \cdot v$. The other generators are similar. We have
 \begin{align*}
   b_1 \cdot v &= b_1 \cdot \left[\left(\sum_{g \in G} v_g\right) \otimes v_2 \otimes \dots 
   \otimes v_n \right] \\
   &= \left( \sum_{g \in G} g \cdot v_2 \otimes v_g \right) \otimes v_3 \otimes \dots 
   \otimes v_n \\
   &= \sum_{g \in G} v_2 \otimes v_g \otimes v_3 \otimes \dots \otimes v_n \\
   &= v_2 \otimes \left( \sum_{g \in G} v_g \right) \otimes v_3 \otimes \dots \otimes v_n 
   = v_2 \otimes v_1 \otimes v_3 \otimes \dots \otimes v_n \, .
 \end{align*}
 Given $\sigma \in S_n$ choose $b \in B_n$ such that $s(b) = \sigma$.
 Using Corollary~\ref{cor:braided}, we have 
\[ 
 x (\sigma \cdot v) = x (s(b) \cdot v) = x (b \cdot v) = x (v) \, ,
\] 
showing that $\igs (x) \in \snhgs$.
\end{proof}

\begin{rem}
  $\igs$ need not preserve the multiplication since $\bn$ on $\hsn$ and
  $\sn$ on $\hgsn$ need not be compatible.
\end{rem}

We mention the following functorial property.

\begin{pro}
  Suppose $\phi : H \to K$ is a morphism of finite dimensional $G$-graded $G$-modules.
  Then we have the following morphisms of algebras with unity
  \begin{align*}
    \phi^\ast : \sym \llbracket K_e^\ast \rrbracket 
    \to \sym \llbracket H_e^\ast \rrbracket \, , \quad
    \phi^\ast : \skgs \to \shgs \, .
  \end{align*}
\end{pro}

\begin{proof}
 Since $\phi$ preserves the $G$-grading, it maps $H_e$ into $K_e$.
 Note that this restriction is also a morphism of $G$-graded $G$-modules.
 Hence the existence of the first map follows from Theorem~\ref{thm:ringmap}.

 Since $\phi$ is $G$-equivariant, it maps $H^G$ into $K^G$.
 Applying Proposition~\ref{pro:symfunctorial}, we have the second map.
\end{proof}

\section{$G$-Frobenius Manifolds}
\label{sec:gfm}

\subsection{Ordinary Frobenius Manifolds}
\label{sec:FM}

We review the concept of formal Frobenius manifold in this section.
See \cite{km}, \cite{manin} for details.

Let $H$ be a finite dimensional vector space, and fix a basis $\{\partial_a\}$.
Let $\{x^a\}$ be its dual basis.
Recall from Section~\ref{ssub:functions} that we have the identification
$\symhs = k\llbracket H^\ast \rrbracket$,
which we interpret as functions on $H$.
We identify the vectors in $H$ with differential operators on $\symhs$.
Then we can regard $\xh :=\symhs \otimes H$ to be global vector fields on $H$.
Suppose also that we have a symmetric nondegenerate bilinear form $g$ on $H$.

Given a function $Y \in \symhs$, we define a multiplication $\circ_Y$ 
on $\xh$ in the following way.
\[
  \partial_a \circ_Y \partial_b := \partial_a \partial_b \partial_k Y g^{kl} \partial_l \, ,
\]
and extending to $\xh$ using the symmetric monoidal structure of vector spaces and
commutative ring structure of $\symhs$. 

Since the order of differentiation does not matter, we see that $\circ_Y$ is commutative.
But $\circ_Y$ need not be associative. 
In fact, its associativity is equivalent to the following 
system of partial differential equations known as WDVV-equations.
We use the notation $Y_a := \partial_a Y$, etc.
\[
  Y_{abk} g^{kl} Y_{lcd} = Y_{bck} g^{kl} Y_{lcd} \, .
\]
\begin{defi}
  $(H,g,Y)$ is a \emph{formal Frobenius manifold} if $Y$ satisfies the WDVV-equations. 
\end{defi}
It follows that a formal Frobenius manifold yields a commutative algebra structure
on $\xh$.

A closely related concept is that of cohomological field theory(CohFT) of 
Kontsevich-Manin\cite{km}.
For a definition of it, see Section~\ref{ssec:gcohft}.
They proved the following theorem.

\begin{thm}
 A formal Frobenius manifold is equivalent to the genus zero
 part of a CohFT. 
\end{thm}

  We outline the proof of one direction of this theorem. 
  Namely, we review how one obtains a formal Frobenius manifold
  out of a CohFT.

  Suppose that we are given a CohFT $\Lambda := (H,g, \{\Lambda_n\})$.
For each $n \geq 3$, the $n$th correlation function $Y_n$ of $\Lambda$ is 
a symmetric $n$-linear form on $H$ defined as
\[
  Y_n : \hn \to k \, , \quad v \mapsto \int_{\mzb{n}} \Lambda_n(v) \, ,
\]
where $\mzb{n}$ denotes the moduli space of rational stable curves with $n$ marked points.
Now, set
\[
  Y := \sum_{n \geq 3} \frac{1}{n!} Y_n \in \symhs \, .
\]
Then the axioms of CohFT and the topology of $\mzb{n}$ together
imply that $Y$ satisfies the WDVV-equations.

\subsection{$G$-Frobenius Algebras}
\label{sec:GFA}
We recall the definition of $G$-Frobenius algebras 
(also known as Turaev algebras, or crossed $G$-algebras) in this section.
See \cite{jkk}, \cite{orbifolding}, \cite{kp}, \cite{ms} and \cite{turaev} for details.

These are  certain algebra structures on $G$-graded $G$-modules.
First, we put the following compatibility condition on $G$-action with
the $G$-grading.
Let $\left( H, \rho \right)$ be a $G$-graded $G$-module.

\begin{defi}
  $(H,\rho)$ is \emph{self-invariant} if $\gamma$ acts on
  $H_{\gamma}$ trivially, for any $\gamma \in G$.
\end{defi}

We consider a bilinear form on $(H, \rho)$ that is compatible with the $G$-action.

\begin{defi}
  A symmetric bilinear form $\eta$ on $(H,\rho)$ is \emph{$G$-invariant} if  
  \begin{equation*}
   \eta(\rho(\gamma)v_1, \rho(\gamma)v_2) = \eta(v_1, v_2) 
  \end{equation*}
  for any $\gamma \in G$ and $v_1, v_2 \in H$.
\end{defi}

Our bilinear form should also be compatible with the $G$-grading in the following sense.
Let $v_g$, $v_h \in H$ be $G$-homogeneous vectors of $G$-degree $g$ and $h$.
\begin{defi}
  A bilinear form $\eta$ on $H$ \emph{preserves the $G$-grading} if
$\eta(v_g, v_h)$ equals zero unless $gh = e$.
\end{defi}

\begin{rem}
  $\eta$ preserves the $G$-grading if and only if
  $\eta \in [(H^\ast)^{\otimes 2}]_e$,
  the $G$-degree $e$ part.
\end{rem}

\begin{rem}
  We note that a symmetric bilinear form on a self-invariant
  $G$-graded $G$-module that preserves the $G$-grading is actually
  a braided bilinear form in the sense of Remark~\ref{rem:braidedmultilinear}.
\end{rem}

We observe the following characterization
about  nondegenerate bilinear forms on $H$.

\begin{lem}
  Suppose $\eta \in [(H^\ast)^{\otimes 2}]_e$.  
  Then $\eta$ is nondegenerate if and only if the following linear map
  \[
    H_g \to [H^\ast]_g = (H_{g^{-1}})^\ast \, ,
    \quad v_g \mapsto \eta(v_g, \bcdot)
  \]
  is an isomorphism for each $g \in G$.
  \label{lem:nondegenerate}
\end{lem}

The following lemma plays an important role in the theory of $G$-Frobenius algebras.

\begin{lem}
 Let $\eta$ be 
 a $G$-invariant nondegenerate symmetric bilinear form that preserves the $G$-grading
 on $H$.
 Then $\igs \eta$ is nondegenerate on $H^G$.
 \label{lem:etag}
\end{lem}

Our main definition in this section follows.
We denote the $G$-degrees of homogeneous elements as subscripts.

\begin{defi}
  A \emph{$G$-Frobenius algebra $((H,\rho), \eta, \cdot, 1)$} 
  is a unital associative algebra on a finite dimensional self-invariant
  $G$-graded $G$-module $(H,\rho)$ 
  with a $G$-invariant nondegenerate symmetric bilinear form  $\eta$
  that preserves the $G$-grading, and
  satisfying the following compatibility conditions:
  \begin{enumerate}
    \item \emph{$G$-equivariance of multiplication}. 
      $\rho(\gamma)v_1 \cdot \rho(\gamma)v_2 =\rho(\gamma) (v_1 \cdot v_2)$ 
      for any $\gamma \in G$ and
      $v_1, v_2 \in H$.
    \item \emph{$G$-graded multiplication}.
      $v_{m_1} \cdot v_{m_2} \in H_{m_1 m_2}$.
    \item \emph{Braided commutativity}.
      $v_{m_1} \cdot v_{m_2} = \rho(m_1^{-1}) v_{m_2} \cdot v_{m_1}$.
    \item \emph{Invariance of the metric}.
      $\eta(v_1 \cdot v_2, v_3) = \eta(v_1, v_2 \cdot v_3)$ for any $v_1, v_2, v_3 \in H$.
    \item \emph{$G$-invariant identity}.
      $\rho(\gamma) 1 = 1$.
     \end{enumerate}
\end{defi}

\begin{rem}
  The standard definition of a $G$-Frobenius algebra requires that 
  it satisfies the \emph{trace axiom}. We do not consider it in this paper.
See \cite[Definition 4.13]{jkk}.
\end{rem}

\begin{rem}
  \cite{orbifolding} shows that a $G$-Frobenius algebra
  can also be defined in terms of certain $G$-braided bi- and tri-linear forms.
  \label{rem:trilinear}
\end{rem}

\begin{rem}
  In \cite{dt, kp} the authors show that a $G$-Frobenius algebra is a monoid object 
in $\dcg$-$\md$, satisfying further conditions.
\label{rem:mod-DCG}
\end{rem}

\begin{rem}
 Note that the identity should be $G$-homogeneous with $G$-degree equal to the 
 identity in $G$.
\end{rem}

\begin{rem}
  When $G$ is trivial, we recover the notion of Frobenius algebra.
  In particular, $H_e$ of any $G$-Frobenius algebra $H$
  has the induced structure of a Frobenius algebra.
  Moreover, \cite[Proposition 2.1]{orbifolding} shows that 
  the restriction of the algebra structure and the bilinear form on 
  $H^G$ also yields a Frobenius algebra structure on it.
\end{rem}

\subsection{$G$-Frobenius Manifolds}

Let $(H, \rho)$ be a self-invariant $G$-graded
$G$-module of finite dimension, $\Phi$ the braiding as in Remark~\ref{rem:mod-DCG}, 
$\eta$ a $G$-invariant nondegenerate symmetric
bilinear form on it that preserves the $G$-grading.
Let $Y^n \in [\mathpzc{Br}^n H^\ast]_e$ for each $n \geq 3$ and 
set $Y := \sum Y^n \in \left[ \mathpzc{Br} \llbracket H^\ast \rrbracket \right]_e$.
Write $\eta_e := \ies \eta$, $\eta^G := \igs \eta$, $Y_e := \ies Y$, and
$Y^G := \igs Y$.

\begin{defi}
  $((H, \rho), \eta, Y)$ is a
  \emph{pre-$G$-Frobenius manifold} if 
  both $Y_e$ and $Y^G$ satisfies the WDVV-equations.
  Moreover, if there is a vector $1 \in H_e$ such that 
  $((H, \rho), \eta, Y^3, 1)$ is a $G$-Frobenius algebra,
  then $((H,\rho), \eta, Y, 1)$ is a \emph{$G$-Frobenius manifold}.
  \label{defi:gfm}
\end{defi}

\begin{rem}
  Here we are using Remark~\ref{rem:trilinear}.  
\end{rem}

\begin{rem}
  $(H_e, \eta_e, Y_e)$ and $(H^G, \eta^G, Y^G)$ become ordinary formal Frobenius manifolds.
  The nondegeneracies of $\eta_e$ and $\eta^G$ are guaranteed by Lemmas~\ref{lem:nondegenerate}
  and \ref{lem:etag}.
\end{rem}

\begin{rem}
Kaufmann envisioned such a structure as an ingredient for generalizing the results of
\cite{orbifolding} to the level of Frobenius manifolds.
I realized it by interpreting $G$-graded $G$-modules as $G$-braided spaces.
\end{rem}

\section{Correlation Functions of $G$-Cohomological Field Theory}
\label{sec:corrgcohft}

\subsection{Moduli Spaces}

We review the theory of moduli spaces of pointed admissible $G$-covers
from \cite{jkk}.
We start with (unpointed) admissible $G$-covers from \cite{acv} and \cite{jkk}.

Let $(C \to T, p_1, \dots, p_n)$ be a stable curve over $T$ of genus zero 
with $n$ marked points $p_1, \dots, p_n$.

\begin{defi}
  A finite morphism $\pi : E \to C$ is an \emph{admissible $G$-cover}
  if it satisfies the following conditions.
  \begin{enumerate}[(i)]
    \item $E/T$ is a nodal curve, possibly disconnected.
    \item Nodes of $E$ maps to nodes of $C$.
    \item $G$ acts on $E$, respecting the fibers.
    \item $\pi$ is a principal $G$-bundle outside of the special points (nodes 
      or marked points).
    \item Over the nodes of $C$, the structure of the maps $E \to C \to T$
      is locally analytically isomorphic to the following:
      \[ \spec A[z,w]/(zw - t) \to \spec A[x,y]/(xy - t^r) \to \spec A \, , \] where $t \in A$, $x = z^r$, and $y = w^r$ for some positive integer $r$.
   \item Over the marked points of $C$, the structure of the maps
     $E \to C \to T$ is locally analytically isomorphic to the following:
     \[ \spec A[z] \to \spec A[x] \to \spec A \, ,
     \]
     where $x = z^s$ for some positive integer $s$.
   \item At each node $q$ of $E$, the eigenvalues of the action of the stabilizer $G_q$
     on the two tangent spaces should be the multiplicative inverses of each other.
    \end{enumerate}
\end{defi}

\begin{rem}
  \cite{acv} shows that the stack of admissible $G$-covers is isomorphic 
  to the stack $\overline{M}_{0,n}(BG)$ of balanced twisted stable maps
  into the classifying stack of $G$, and is a smooth Deligne-Mumford stack,
  flat, proper, and quasi-finite over $\overline{M}_{0,n}$. 
\end{rem}

As in the case of principal $G$-bundles, admissible $G$-covers are
classified by their holonomies up to the adjoint action of $G$.
Let $C_{\text{gen}}$ and $E_{\text{gen}}$ be the points that are
neither nodes nor marked points in $C$ and $E$. 
Choose a point $p_0 \in C_{\text{gen}}$.

\begin{pro}
  If $C_{\text{gen}}$ is connected, then 
  there is a one to one correspondence between the set of homomorphisms 
 $\pi_1(C_\text{gen},p_0) \to G^n$ and the set of isomorphism classes 
 of admissible $G$-covers over $C$ together
 with a point $\tilde{p}_0$ in the fiber of $p_0$.
 \label{pro:holonomy}
 \end{pro}
 
 Let $\overline{G}$ be the set of conjugacy classes of $G$, 
 and $\overline{\boldsymbol \gamma} \in \overline{G}^n$ be
 the $n$-tuple of conjugacy classes determined by $\boldsymbol \gamma \in G^n$.
 \cite{jk} shows that we have the following decomposition:
\[ \overline{M}_{0,n}(BG) 
  = \coprod_{\overline{\boldsymbol \gamma} \in \overline{G}^n} 
  \overline{M}_{0,n}(BG;\overline{\boldsymbol \gamma}) \, ,
\]
where $\overline{M}_{0,n}(BG;\overline{\boldsymbol \gamma})$ denote
the substack with holonomies that can be labeled by $\overline{\boldsymbol \gamma}$.
Note that $\overline{M}_{0,n}(BG;\overline{\boldsymbol \gamma})$ 
consists of multiple components.

We now come to the pointed version of the previous moduli space.

\begin{defi}
  For a $g=0$ stable curve with $n \geq 3$ marked points 
  $(C \to T , p_1, \dots, p_n)$, an \emph{$n$-pointed admissible $G$-cover} is
  an admissible $G$-cover $\pi : E \to C$ with $n$ marked points 
  $\tilde{p_i} \in E$ such that $\pi(\tilde{p_i}) = p_i$.
  The morphisms in the \emph{stack of admissible $G$-covers
    $\overline{M}^G_{0,n}$} are the $G$-equivariant fibered diagrams
    preserving the points $\tilde{p}_i$.
\end{defi}

\begin{rem}
We have forgetful morphisms
\[
  \begin{CD}
 \mgzb{n} @>\stt>> \mzb{n}(BG) @>\sth>> \mzb{n} \, .
  \end{CD}
\]
Let $\st := \sth \circ \stt$.
  \cite{jkk} shows that $\overline{M}^G_{0,n}$ are smooth Deligne-Mumford stacks, 
  flat, proper, and quasi-finite over $\overline{M}_{0,n}$.
  \end{rem}
  
Given a pointed admissible $G$-cover, 
let $\boldsymbol m := (m_1, \dots, m_n)$ be the monodromy around
the marked points on the cover, and $\overline{M}^G_{0,n}(\boldsymbol m)$
denote the substack with the monodromy $\boldsymbol m$.
Then we have the decomposition
\begin{align*}
  \overline{M}^G_{0,n} = \coprod_{\boldsymbol m \in G^n}
  \overline{M}^G_{0,n}(\boldsymbol{m}) \, .
\end{align*}

Given a pointed admissible $G$-cover, the structure of the fiber
over a marked point on the underlying curve with monodromy $m$
is the same as $G/<m>$ as a $G$-set.
This gives rise to the $G^n$-action on $\mgzb{n}$.
Also, we have the $S_n$-action on it given by switching the labels
of the marked points on underlying curves.

We now describe two morphisms  that are relevant in $g = 0$: the 
\emph{forgetting tails morphism} $\tau$ and the \emph{gluing morphism} $gl$.

$\tau$ is defined when the monodromy of the last point is trivial. Suppose
$\bm \in G^n$. Then
\[
  \tau: \mgzb{n+1}(\bm, e) \to \mgzb{n}(\bm)
\]
is given simply by forgetting the last marks on the underlying stable curve
and the point over it. If we get an unstable curve as the result, 
we stabilize it and \cite{jkk} shows that we have a way to 
define a suitable pointed admissible $G$-cover on the stabilized curve.

We can glue two $G$-covers at marked points when their monodromies are inverse 
to each other. Let $\mu \in G$ and $n = n_1 + n_2$. Then
\[
  gl: \mgzb{n_1 + 1}(\bm , \mu) \times \mgzb{n_2 + 1}(\mu^{-1}, \bm') \to 
  \mgzb{n}(\bm, \bm')
\]
is given by gluing the underlying curve at the last marked point of the first curve
and the first marked point of the second curve, and gluing the $G$-cover equivariantly
at the points over the marked points just glued. The result is already a
pointed admissible $G$-cover.
More generally, the gluing morphism can be defined
for any partition $I \sqcup J$ of $\{ 1, \dots, n\}$  as
\[
gl: \mgzb{n_1 + 1}(\bm_I) \times \mgzb{n_2 + 1}(\bm_J) \to 
  \mgzb{n}(\bm)
\]
making use of the symmetric group action.
Here $\bm_I$ and $\bm_J$ are obtained by adjoining $\mu$ and $\mu^{-1}$ at some position
to the corresponding components of $\bm$.

Next, we define distinguished components of $\mgzb{n}$ following \cite{jkk}.
Suppose that we have $n$ marked points at $ P := \left\{ e^{\frac{2k \pi i}{n}} 
: 1 \leq k \leq n \right\}$ 
on $\mathbb{P}^1 = \C \cup \left\{ \infty \right\}$.
We label the points from $1$ to $n$ counterclockwise starting from the point $1 \in \C$.
Draw line segments from $0$ to each marked point.
Then this determines a set of generators of the fundamental group
of $\mathbb{P}^1 - P$ at $0$.
Let $\bm \in G^n$ with $\prod_i m_i = e$.
Then $\bm$ determines a holonomy, and
Proposition~\ref{pro:holonomy} gives us an admissible $G$-cover and
a point $p$ over $0$.
Parallel transporting $p$ along the line segments above yields
$n$ points over the marked points on the underlying $\mathbb{P}^1$.
Note that the monodromy at these points are also given by $\bm$.
We denote $[\bm]$ the component of $\mgzb{n}(\bm)$ that contains the point
representing this pointed admissible $G$-cover.

\begin{rem}
The homology classes of $\mgzb{n}$ in this paper denote
those of its coarse moduli space. Namely, we do not divide by
the orders of automorphism groups.
\label{rem:homology}
\end{rem}

\subsection{$G$-Cohomological Field Theory}
\label{ssec:gcohft}

In this section, we recall the notion of $G$-cohomological
Field Theory($G$-CohFT) from \cite{jkk}, and review how
one obtains an ordinary cohomological field theory(CohFT)
and a $G$-Frobenius algebra
out of it.
We consider only $g=0$ case.

\begin{defi}
  A genus zero $G$-CohFT is 
  a quadruple $((H,\rho), \eta, \{\Lambda_n\},1)$
  where
  \begin{enumerate}[(i)]
    \item $(H,\rho)$ is a $G$-graded $G$-module,
    \item $\Lambda_n \in \bigoplus_{\boldsymbol m} 
      H^\bullet (\overline{M}^G_n(\boldsymbol m)) \otimes H^\ast_{\boldsymbol m}$
      and $G^n \rtimes S_n$-invariant,
    \item $1 \neq 0 \in H_e$ such that
      \begin{enumerate}[(a)]
	\item $1$ is $G$-invariant,
	\item $\Lambda_{n+1} (v_{\boldsymbol m}, 1) = \tau^\ast \Lambda_n(v_{\boldsymbol m})$
	  under the forgetting tails morphism $\tau$,
      \end{enumerate}
    \item $\eta$ is a nondegenerate symmetric bilinear form on $H$ such that
      \begin{align}
		\eta(v_{m_1}, v_{m_2}) 
  := \int_{[m_1, m_2, 1]} \Lambda_3(v_{m_1}, v_{m_2}, 1) \, ,
  \label{eq:id}
      \end{align}
    \item for any basis $\{e_\alpha\}$ of $H$,
      \begin{align*}
	gl^\ast \Lambda_n(v_{\boldsymbol m}) 
	= \sum_{\alpha , \beta} \Lambda_{n_1 + 1}(v_{\boldsymbol m_I}, e_\alpha)
	\eta^{\alpha \beta} \Lambda_{n_2 + 1} (e_\beta, v_{\boldsymbol m_J}) 	
      \end{align*}
      under the gluing morphism $gl$,
    for any partition $I \sqcup J$ of $\{1, \dots, n\}$.
  \end{enumerate}
\end{defi}

\begin{rem}
  The notion of CohFT as defined in \cite{km} does not involves the axioms
  regarding $1$. Hence
  we recover the notion of the $g=0$ part of the CohFT
  by considering trivial group $G$ and removing the axiom (iii) above and
  Equation~\eqref{eq:id}.
\end{rem}

We now review how to obtain an ordinary CohFT from a $G$-CohFT. 
This is a two-step process that involves the three moduli spaces
$\mgzb{n}$, $\mzb{n}(BG)$ and $\mzb{n}$.

The first step involves the forgetful map
\[
  \stt[\overline{\bm}] : \mgzb{n}(\overline{\bm}) \to \mzb{n}(BG; \overline{\bm}) \, ,
\]
where
\[
  \mgzb{n}(\overline{\bm}) := \coprod_{\bm \in \overline{\bm}} \mgzb{n}(\bm) \, .
\]
We have the following proposition.

\begin{pro}
  There are unique classes 
 \[ 
   \widehat{\Lambda}_n \in \bigoplus_{\overline{\bm} \in \overline{G}^n} 
   H^\bullet(\mzb{n}(BG;\overline{\bm})) \otimes (H^G_{\overline \bm})^\ast 
\]
such that 
\[
\stt[\bar{\bm}]^\ast \widehat{\Lambda}_n(v_{\bar{\bm}}) = \Lambda_n(v_{\bar{\bm}})
\]
for all $v_{\bar{\bm}} \in H^G_{\bar{\bm}}$.
\end{pro}

Once we obtain $\widehat{\Lambda}_n$, we define $\bar{\Lambda}_n$ in the following way.

\begin{defi}
  \[
    \bar{\Lambda}_n := 
    \sth[\ast] \widehat{\Lambda}_n \in H^\bullet(\mzb{n}) \otimes \hgs^{\otimes n} \, .
  \]
\end{defi}

Let $\eta^G$ be the restriction of $\eta$ to $H^G$. Then we get the CohFT we wanted
as follows.

\begin{thm}
  $(H^G, \eta^G, \{\overline{\Lambda}_n\})$ is a CohFT.
\end{thm}

\begin{rem}
  In \cite{jkk}, the authors use $\bar{\eta} : = (1/|G|)\eta^G$ instead of
  $\eta^G$. This is because they include $1$ as a data for the definition of
  a CohFT. Once we remove the conditions about $1$, we can normalize $\eta^G$ 
  as needed, and we still get a CohFT.
\end{rem}

Next, we consider the $G$-Frobenius algebra 
that is contained in any $G$-CohFT $((H,\rho),\eta, \Lambda_n, 1)$.
Let $g$, $h$, and $k$ be in $G$ such that $g h k = e$. 
Fix a basis $\left\{ e_{\alpha} \right\}$ for $H_k$ and 
$\left\{ f_{\beta} \right\}$ for $H_{k^{-1}}$. 
Let $v \in H_{g}$ and $w \in H_h$. 
Define 
\begin{align}
  v \cdot w := \int_{\left[ g, h, k \right]} \Lambda_3(v, w, e_\alpha) 
  \eta^{\alpha \beta} f_\beta \, .
  \label{eq:gfa}
\end{align}

\begin{thm}
  $((H,\rho), \eta, \cdot, 1)$ is a $G$-Frobenius algebra.
  \label{thm:gfrob}
\end{thm}

\begin{rem}
  In particular, $\rho$ is self-invariant, and $\eta$ is $G$-invariant and 
  preserves the $G$-grading.
  \label{rem:gfa}
\end{rem}

\begin{rem}
  Due to our convention in Remark~\ref{rem:homology} on the homology classes in $\mgzb{n}$,
  the homology class that is used for the definition of $\cdot$
  in \cite{jkk} is the same class as we use here.
\end{rem}

 \subsection{Correlation Functions}
 \label{sec:GcohFT}

 In this section, we define two types of correlation functions that we can get
 out of a $G$-CohFT $((H, \rho), \eta, \{\Lambda_n\}, 1)$, 
 a symmetric one and a braided one, and prove that they coincide on $H^G$.

Given $\bm = (m_1, \dots, m_n) \in G^n$, 
 let $\tau$ be a component of $\mzb{n}(BG; \overline{\bm})$.
 Then $\stt[\bar{\bm}] ^{-1}(\tau)$ is the union of the components of 
 $\mgzb{n}(\overline{\bm})$ that are mapped to
 $\tau$. Given a component $\kappa$ of $\mgzb{n}(\overline{\bm})$,
 let $\tau(\kappa)$ denote the component of $\mzb{n}(BG;\overline{\bm})$ 
 that contains the image of $\kappa$ under $\stt[\bar{\bm}]$.
 Set $u(\kappa) := \stt[\bar{\bm}]^{-1}(\tau(\kappa))$ and define
 $p(\kappa)$ to be the degree of the forgetful map
 \[
   \stt[\bar{\bm}] |_{u(\kappa)} : u(\kappa) \to \tau(\kappa) \, .
\]
In particular, we write $p(\bm) := p([\bm])$.
Also, since $\bg : [\bm] \to \bg[\bm]$ is an isomorphism 
that is compatible with the forgetful map, we have
$p(\bg[\bm]) = p(\bm)$. 
Let $\left\{ \kappa_i(\bm) \right\}$ be the set of components of
$\mgzb{n}(\bm)$.
This is a finite set since any component in $\mgzb{n}$
can be written as $\bg[\boldsymbol m']$ for some 
$\bg$ and $\bm'$ with $\prod m_i' = e$ in $G^n$.

 \begin{defi}
   Given a $G$-CohFT $((H,\rho), \eta, \{\Lambda_n \}, 1)$, its \emph{degree $n$ symmetric 
   correlation function $Y^n_S$} is defined as
   \[ Y^n_S(v) := \sum_i \frac{1}{p(\kappa_i(\bm))}
   \int_{\kappa_i(\boldsymbol m)} \Lambda_n (v) \, ,\]
   for any $v \in H_{\bm}$.
 \end{defi}

 \begin{pro}
   $Y^n_S$ is a symmetric $n$-linear form.
   \label{pro:symmetric}
 \end{pro}

 \begin{proof}
It is enough to consider homogeneous vectors.
Suppose $v \in H_{\boldsymbol m}$ and $\sigma \in S_n$.
\begin{align*}
  Y^n_S(\sigma v) &= \sum_i \frac{1}{p(\kappa_i(\sigma \bm))}
  \int _{\kappa_i(\sigma \bm)} \Lambda_n (\sigma v) \\
  &= \sum_i \frac{1}{p(\kappa_i(\sigma \bm))} 
  \int_{\kappa_i (\sigma \bm)} \sigma \Lambda_n(v) \quad 
  \text{by the $G^n \rtimes S_n$ invariance,} \\
  &= \sum_i \frac{1}{p(\kappa_i (\sigma \bm))}
  \int_{\sigma^{-1} \kappa_i (\sigma \bm)} \Lambda_n(v) \, .
\end{align*}

Note that we have the following commutative diagram, where the two $\sigma$'s are
isomorphisms.
\[
  \begin{CD}
    u(\kappa_j(\bm)) @>\sigma>>  u(\kappa_i(\sigma \bm)) \\
@VV\sttilde_{\bar{\bm}}V   @VV\sttilde_{\bar{\sigma \bm}}V \\ 
\tau(\kappa_j(\bm)) @>\sigma>> \tau(\kappa_i(\sigma \bm))
  \end{CD}
\]
This implies that $p(\kappa_i(\sigma \bm)) = p(\kappa_j(\bm))$.
Hence,
\begin{align*}
  Y^n_S(\sigma v) 
  = \sum_j \frac{1}{p(\kappa_j(\bm))}
  \int_{\kappa_j (\bm)} \Lambda_n(v)
  = Y^n_S(v) \ .
\end{align*}
   \end{proof}

Since the ways of representing components 
as $\bg[\bm']$
are not unique,
let $\nu(\kappa)$ be the number of ways that $\kappa$ can be written
in this way.
In particular, let $\nu(\bm) := \nu([\bm])$.
Note that $\nu$ is invariant under the $B_n$-action,
and $\nu(\bg [\boldsymbol m]) = \nu (\boldsymbol m)$
for any $\bg \in G^n$ and $\boldsymbol m$,
since these actions are isomorphisms.
Also, we see that
the component $\bg[\bm]$ has the constant monodromy
$md(\bg[\bm]) := \left( \gamma_1 m_1 \gamma_1^{-1} , 
\dots, \gamma_n m_n \gamma_n^{-1} \right)$.

\begin{defi}
  The \emph{degree $n$ braided correlation function $Y^n_B$} of 
  a $G$-CohFT $((H,\rho), \eta, \{\Lambda_n \}, 1)$ is defined as
  \[
    Y^n_B(v) := 
    \sum_{\prod \boldsymbol m = e}
    \frac{|G|^n}{p(\boldsymbol m) \nu(\boldsymbol m)}
    \int_{[\boldsymbol m]} \Lambda(v) \, ,
  \]
  for any $v \in H^{\otimes n}$.
  \label{defi:ynb}
\end{defi}

\begin{rem}
  Note that $Y^n_B \in (\brnhs)_e$ for each $n$.
\end{rem}

The analogue of Proposition~\ref{pro:symmetric} follows.

\begin{pro}
  $Y^n_B$ is a braided $n$-linear form.
\end{pro}

\begin{proof}
  Note that the $B_n$ action on $\mzb{n}(BG)$ amounts to
  choosing a new set of generators of the fundamental 
  group on the underlying curve, to determine the holonomy 
  of its admissible $G$-cover.
Using the $B_n$-invariance of $\nu$, and the
fact that $b[\boldsymbol m] = [b \boldsymbol m]$,
the proof is completely analogous to that of Proposition~\ref{pro:symmetric}.
\end{proof}

Note that $Y^n_B$ restricted to $(H^G)^{\otimes n}$ are symmetric $n$-linear forms.
Moreover, we have the following proposition.

\begin{pro}
  For $v \in (H^G)^{\otimes n}$, 
  \[
Y^n_S(v) = Y^n_B(v) \, .
  \]
  \label{pro:symbr}
\end{pro}

\begin{proof}
  We first prove Equation~\eqref{eq:bsrelation} below, that relates 
  $Y^n_S$ and $Y^n_B$ for any homogeneous vector $w$.
  Hence, suppose $w \in H_{\boldsymbol m}$.
  Then we have
  \begin{align*}
    Y^n_S(w) 
    &= \sum_i \frac{1}{p(\kappa_i(\boldsymbol m))} 
    \int_{\kappa_i(\boldsymbol m)} \Lambda(w) \\
    &=  \sum_{\substack{\prod \boldsymbol m' = e \\ 
    md(\boldsymbol \gamma[\boldsymbol m'])=\boldsymbol m }}
    \frac{1}{\nu(\boldsymbol m')}
    \frac{1}{p(\bg[\bm'])}
    \int_{\boldsymbol \gamma [\boldsymbol m']} \Lambda(w).
  \end{align*}
  By the $G^n \rtimes S_n$-invariance of $\Lambda_n$ 
  and $G^n$-invariance of $p$,
\[
  Y^n_S(w) =   \sum_{\substack{\prod \boldsymbol m' = e \\ 
    md(\boldsymbol \gamma[\boldsymbol m'])=\boldsymbol m }}
    \frac{1}{\nu(\boldsymbol m')}
       \frac{1}{p(\bm')}
 \int_{[\boldsymbol m']} \Lambda(\boldsymbol \gamma^{-1} w).
\]
By the definition of $Y^n_B$,
\[
Y^n_S(w) = \sum_{\substack{\prod \boldsymbol m' = e \\ 
    md(\boldsymbol \gamma[\boldsymbol m'])=\boldsymbol m }}
    \frac{1}{|G|^n} Y^n_B (\boldsymbol \gamma^{-1} w).
\]
Since $Y^n_B$ is zero on the vectors of total degree not equal to $e$,
\begin{align}
  Y^n_S(w) = \frac{1}{|G|^n} 
  Y^n_B(\sum_{\boldsymbol \gamma \in G^n}\boldsymbol \gamma w) \, .
  \label{eq:bsrelation}
\end{align}

Now we use this equation to prove the proposition.
Let $c(\gamma)$ be the number of elements in the conjugacy class of $\gamma$.
For $\boldsymbol m \in G^n$, define
$c(\boldsymbol m) := \prod c(m_i)$.
Suppose $v \in H^G_{\overline{\boldsymbol m}}$ and
let \[v = \sum_{i = 1} ^{c(\boldsymbol m)} w_i\] be the homogeneous decomposition of $v$.
Then
\[
  |G|^n v = \sum_{\boldsymbol \gamma \in G^n} \boldsymbol \gamma v = 
  \sum_{\boldsymbol \gamma \in G^n} \sum_{i = 1} ^{c(\boldsymbol m)} \boldsymbol \gamma w_i\, .
\]
Hence we conclude
\begin{align*}
  Y^n_S(v) = Y^n_S(\sum_{i=1}^{c(\boldsymbol m)} w_i) &= 
  \sum_{i=1}^{c(\boldsymbol m)} Y^n_S(w_i)\\ 
  =& \sum_{i=1}^{c(\boldsymbol m)} \frac{1}{|G|^n} 
  Y^n_B(\sum_{\boldsymbol \gamma \in G^n} \boldsymbol \gamma w_i) 
= Y^n_B(v) \, .
\end{align*}
\end{proof}
 
 \subsection{Pre-$G$-Frobenius Manifold Structure}
 \label{sec:gcohft}

In this section, we prove that any $G$-cohomological field theory
gives rise to a pre-$G$-Frobenius manifold structure on its state space.
We first make the following observation.
Let $H$ be a finite dimensional vector space.
\begin{lem}
  If $Y := \sum Y^n \in \symhs$ satisfies the WDVV-equations,
  then $Y' := \sum a^n Y^n$  for some $a \in k$ also satisfies the WDVV-equations.
  \label{lem:modified}
\end{lem}

\begin{proof}
  This is true since WDVV-equations can be proved degree-by-degree.
\end{proof}

Given a $G$-CohFT $((H,\rho), \eta, \{\Lambda_n \}, 1)$,
set \[Y_B := \sum_{n \geq 3} \frac{1}{n!}Y_B^n \, .\]

\begin{thm}
$((H, \rho), \eta, Y_B)$ is a pre-$G$-Frobenius manifold.
\label{thm:cor}
\end{thm}

\begin{proof}
  By Remark~\ref{rem:gfa}, $\rho$ is self-invariant, and $\eta$ is 
  $G$-invariant and preserves the $G$-grading.
  
We first consider $(i^G)^\ast(Y_B)$.
Let $(H^G, \eta^G, \overline{\Lambda}_n)$ be the CohFT
obtained by taking the $G$-quotient of our $G$-CohFT.
Recall that its correlation functions are defined as
\[
  (Y^G)^n(v):= \int_{\overline{M}_{0,n}} \overline{\Lambda}(v)
\]
for any $v \in (H^G)^{\otimes n}$.
In view of the equivalence of CohFT's and Frobenius manifolds,
and Proposition~\ref{pro:symbr}, it is enough to show
\[
Y^n_S(v) = (Y^G)^n(v) 
\]
for any $v \in (H^G)^{\otimes n}$.
Suppose $v$ is homogeneous of $\overline{G}$-degree equal to $\overline{\boldsymbol m}$,
and $v = \sum_{\bm \in \bar{\bm}} w_{\bm}$ be its homogeneous decomposition.
Since $\Lambda_n(w_{\bm})$ is supported on $\mgzb{n}(\bm)$, we have
\begin{align*}
  Y^n_S(v) 
  &= \sum_{\bm \in \bar{\bm}} \sum_i \frac{1}{p(\kappa_i(\bm))}
  \int_{\kappa_i(\bm)} \Lambda_n(w_{\bm}) \\
  &= \sum_{\bm \in \bar{\bm}} \sum_i \frac{1}{p(\kappa_i(\bm))}
  \int_{\kappa_i(\bm)} \Lambda_n(v) 
= \int_{\sum_{\bm \in \bar{\bm}} \sum_i \frac{1}{p(\kappa_i(\bm))} \kappa_i(\bm)}
\Lambda_n(v) \, .
\end{align*}
Using the projection formula, we have
\[
  Y^n_{S}(v) 
  = \int_{\overline{M}(BG;\overline{\boldsymbol m})}
\widehat{\Lambda}_n(v) \, .
\]
Using the Gysin map, we conclude that
\[
  Y^n_S(v) = \int_{\overline{M}_{0,n}} \overline{\Lambda}(v)
  = (Y^G)^n(v) \, .
\]

Consider $(i_e)^\ast(Y_B)$ next.
Let $\b{e} := (e, \dots, e)$, and $v \in H_{\b{e}}$.
Note that $\nu(\b{e}) = |G|$ since
$[\b{e}]$ can be also written as $\b{g}[\b{e}]$ where $\b{g} :=(g, \dots, g)$,
for any $g \in G$.
Since $p(\b{e})=1$, we have 
\[
  Y^n_B (v) = |G|^{n-1} \int_{[\b{e}]} \Lambda_n(v) \, .
\]
Observe that \[ Y' := \sum_n \frac{1}{|G|^{n-1}} \cdot \frac{1}{n!} Y^n_B \] satisfies the
WDVV-equations since $[\b{e}]$ is isomorphic to $\mzb{n}$ and
we are reduced to an ordinary CohFT. 
A simple modification of Lemma~\ref{lem:modified}
implies that $(i_e)^\ast(Y_B)$ satisfies the WDVV-equations.
\end{proof}

\begin{rem}
In fact, we can say a little bit more. By comparing 
Equation~\eqref{eq:gfa} and Definition~\ref{defi:ynb},
we notice that the 
multiplication structure of the $G$-Frobenius algebra
that is contained in the $G$-CohFT differs by constants 
from the one given by the degree $3$ braided correlation functions,
and these constants depend on the $G$-degrees of the vectors multiplied together.
This means that we can identify the two multiplication rules by
rescaling the metric on the underlying $G$-graded $G$-module.
In other words, a $G$-CohFT contains a $G$-Frobenius manifold
up to a rescaling of its metric.
\label{rem:scaling}
\end{rem}

\begin{rem}
  One can ask if it is possible to get a $G$-CohFT starting from a pre-$G$-Frobenius manifold,
  in analogy with the case of ordinary CohFT's and Frobenius manifolds.
  A difficulty answering this question is that the topology of $\mgzb{n}$ is
  not very well understood compared to that of $\mzb{n}$.
\end{rem}

\section{Examples of $\zz$-Frobenius Manifolds}
 \label{sec:example}

 \subsection{$G = \zz$}
\label{sec:z2}

In this section, we specialize to the case $G = \mathbb{Z}/2\mathbb{Z}=\{e,g\}$ and
prove the structure theorem~\ref{thm:structure} for pre-$\zz$-Frobenius manifolds.

We first state some general facts about $\mathbb{Z}/2\mathbb{Z}$-graded
$\mathbb{Z}/2\mathbb{Z}$-modules.

\begin{pro}
  Let $(H, \rho)$ be a $\zz$-graded $\zz$-module that is self-invariant.
  Then $H = H_e \oplus H_g$ has the following properties.
  \begin{enumerate}[(i)]
    \item The non-twisted sector $H_e$ has the following decomposition:
      \[H_e = H_i \bigoplus H_v \ ,\]
where $H_i$ is invariant under $\mathbb{Z}/2\mathbb{Z}$ and
$\rho(g)v = -v$ for any $v \in H_v$.
\item The $G$-invariants decomposes as
  \[H^G = H_i \bigoplus H_g \ .\]
  \end{enumerate}
  \label{pro:z2FA}
\end{pro}

\begin{proof}
  (i) follows from the theory of representations of finite groups,
  since $H_e$ is a $\zz$-module.
  See, for example, \cite{serre}.
  For (ii), note that the conjugacy classes of $\zz$ are $\zz$ itself,
  and $H_g$ is $\zz$-invariant.
  Also, $(H_e)^G = H_i$. 
\end{proof}

Fix a homogeneous basis $\left\{ x^p \right\}$ of $H^\ast$ under the grading
\[H^\ast = H^\ast_i \bigoplus H^\ast_v \bigoplus H^\ast_g \, , \]
and denote their homogeneous $G$-degree as subscripts.
We have the following description of the  braided multilinear forms on $H$.

\begin{pro}
Given a basis of $H^\ast$ as above, we have the isomorphism of vector spaces
\begin{align*}
  \mathpzc{Br} \llbracket H^\ast \rrbracket \cong k \llbracket x^p \rrbracket / \mathcal{I} \ ,
\end{align*}
where $\mathcal{I}$ is generated by $\{x^p_v x^q_g\}$.
\label{pro:z2zbraid}
\end{pro}
 
\begin{proof}
  Suppose $X \in \brnhs$. Then it is symmetric on $H^G$ and $H_e$ by 
  Propositions~\ref{pro:untwisted} and \ref{pro:invarsym}.
  In view of Proposition~\ref{pro:z2FA}, it remains to consider the value of 
  $X$ on vectors $w \in \hn$ that has at least one factor $w_g \in H_g$ and
  another $w_v \in H_v$. 
  By linearity it is enough to assume that $w$ has only one term.
  Then $w$ can be written as one of the following two forms.

  \[
    w = \dots \otimes w_v \otimes w_{i_1} \otimes \dots \otimes w_{i_k}
    \otimes w_g \otimes \dots
  \]
  or
  \[
 w = \dots \otimes w_g \otimes w_{i_1} \otimes \dots \otimes w_{i_l}
    \otimes w_v \otimes \dots
  \]
  where $w_{i_j}$ denotes vectors in $H_i$.

  We use Corollary~\ref{cor:braided} and consider only the first case. 
  The other one is similar.
  \begin{align*}
    X(w) &= X(\dots \otimes w_v \otimes w_{i_{1}} \otimes \dots \otimes w_{i_k}
    \otimes w_g \otimes \dots) \\
    &= X(\dots \otimes w_v \otimes w_g \otimes \dots) 
  \end{align*}
  by $B_n$ invariance
    and the fact that $H_i$ is invariant under the action of $\zz$. 
Since $w_v$ is in $H_e$, we have
\begin{align*}
  X(w) &= X(\dots \otimes w_g \otimes w_v \otimes \dots) \\
       &= -X(\dots \otimes w_v \otimes w_g \dots) = -X(w)
\end{align*}
since $w_v$ is in the eigenspace of $g$ with eigenvalue $-1$.
Hence we have 
\[
X(w) = 0 
\]
if $w$ has factors of $H_v$ and $H_e$ together.
This implies that any element in $\brnhs$ does not have any term
that has both $x_v^p$ and $x_g^q$ as its factors.
The proposition follows by the identification of symmetric $n$-linear forms and
homogeneous polynomials of degree $n$.
\end{proof}

Once we have a pre-$\zz$-Frobenius manifold structure on $H=H_i \oplus H_v \oplus H_g$, 
we have ordinary Frobenius manifold structures on $H_e = H_i \oplus H_v$ and 
$H^G = H_i \oplus H_g$ with induced potentials. On the other hand, we 
may be able to produce a pre-$\zz$-Frobenius manifold
starting from two ordinary Frobenius manifolds in the following way.

Let $(H_e, \eta_e, Y_e)$ and $(H^G, \eta^G, Y^G)$ be formal Frobenius manifolds.
Suppose that there is a vector space $H_i$ and two linear injections $\iota_e : H_i \to H_e$
and $\iota^G : H_i \to H^G$ such that $(\iota_e)^\ast Y_e = (\iota^G)^\ast Y^G := Y_i$, 
and $(\iota_e)^\ast \eta_e = (\iota^G)^\ast \eta^G := \eta_i$. 
Fix decompositions $H_e = H_i \oplus H_v$ and $H^G = H_i \oplus H_g$. Write
$Y_e = Y_i + Y_v$ and $Y^G = Y_i + Y_g$, and
$\eta_e = \eta_i + \eta_v$ and $\eta^G = \eta_i + \eta_g$
where $(\iota_e)^\ast Y_v = (\iota^G)^\ast Y_g = 0$
and $(\iota_e)^\ast \eta_v = (\iota^G)^\ast \eta_g = 0$.
Let $H := H_i \oplus H_v \oplus H_g$ be a $\mathbb{Z}/2\mathbb{Z}$-graded
$\mathbb{Z}/2\mathbb{Z}$-module as in Proposition~\ref{pro:z2FA}.
Set $Y := Y_i + Y_v + Y_g$ and $\eta := \eta_i + \eta_v + \eta_g$,
extending by zero.

Let $S:= \left\{ i,v,g \right\}$ be the set of the subscript
in the decomposition of $H$ above. 
Then we can talk about the $S^2$-degrees of homogeneous terms of $\eta$.

\begin{thm}[Structure of pre-$\zz$-Frobenius manifolds]
  Suppose that $Y \in [\ths]_e$, and
  $\eta_i$, $\eta_v$, and $\eta_g$ are $S^2$-homogeneous of
  degrees $(i,i)$, $(v,v)$ and $(g,g)$ respectively.
Then
$((H, \rho), \eta, Y)$ is a pre-$\zz$-Frobenius manifold.
Conversely, the $\eta$ and $Y$ of any pre-$\zz$-Frobenius manifold
$((H,\rho), \eta, Y)$
is of this form,
and uniquely determined by the metrics and potentials of the 
two Frobenius manifolds it contains.
\label{thm:structure}
\end{thm}

\begin{proof}
  We first prove that $((H,\rho), \eta, Y)$ constructed as above
  is a pre-$\zz$-Frobenius manifold.
By construction, $\rho$ is self-invariant, and $\eta$ is $\zz$-invariant,
symmetric, and preserves the $\zz$-grading.
Since the $S^2$-degree of $\eta_g$ is $(g,g)$, the condition of 
Lemma~\ref{lem:nondegenerate} is satisfied for $\zz$-degree $e$
by the nondegeneracy of $\eta_e$. The $S^2$-degrees of $\eta_i$
and $\eta_g$, and the nondegeneracy of $\eta^G$ also implies 
that of $\eta_g$, satisfying the condition of the Lemma for $g$.
It follows that $\eta$ is nondegenerate.

Proposition~\ref{pro:z2zbraid} implies that $Y$ is automatically in
$\brhs$.
It remains to show that $Y_e = \ies Y$ and $Y^G = \igs Y$.
Since we are extending by zero and $(\iota^G)^\ast Y_g = 0$, we have
$\ies Y_g = 0$.
With the completely analogous reasoning, we also have $\igs Y_v = 0$.

To prove the converse, we first consider $\eta$.
Writing the $S^2$-degrees as subscripts, we have the decomposition
\[
  \eta = \eta_{ii} + \eta_{iv} + \eta_{ig} + \eta_{vv} + \eta_{vg} + \eta_{gg} \, .
\]
Since $\eta$ preserves the $G$-degree, we have $\eta_{ig} = \eta_{vg} = 0$.
To see that $\eta_{iv} = 0$, take $w_i \in H_i$ and $w_v \in H_v$.
Using the $\zz$-invariance of $\eta$, we have
\[
  \eta_{iv}(w_i , w_v) = \eta_{iv}(g \cdot w_i, g \cdot w_v) 
  = \eta_{iv}(w_i, -w_v) = - \eta_{iv}(w_i, w_v) \, .
\]
For $Y$, let $Y_i$ be the sum of terms that only have factors of $S$-degree $i$,
$Y_v$ the sum of terms that have at least a factor of $S$-degree $v$,
and $Y_g$ the sum of terms that have at least a factor of $S$-degree $g$.
Proposition~\ref{pro:z2zbraid} implies that
every term of $Y$ is uniquely a term of either $Y_i$, $Y_v$ or $Y_g$.
Hence we have the unique decomposition of $Y$ as
\[
Y = Y_i + Y_v + Y_g \, .
\]
Note that $\ies Y_g = \igs Y_v =0$, and $(\iota_e)^\ast Y_v = (\iota^G)^\ast Y_g = 0$
by construction.

Uniqueness is obvious from the decompositions of $\eta$ and $Y$.
\end{proof}

 \subsection{$A_n$ and $D_n$ Singularities and Their Milnor Rings} 
We review the concept of Milnor rings and their
Frobenius algebra structures for certain singularities.
See \cite{singularity} for details.

 We consider two series of polynomials called $A_n$ and $D_n$.
\begin{align*}
  A_{n} &:\frac{1}{n+1} z^{n+1} \\
  D_n &:\frac{1}{2} xy^2 +\frac{1}{2n-2} x^{n-1}
  \label{andn}
\end{align*}
Note that these polynomials define complex valued functions on $\mathbb{C}$
or $\mathbb{C}^2$, and they have isolated singularities (critical points) 
at the origin.

Let $f$ be one of the polynomials. Then its 
\emph{Milnor ring(Jacobian ring or local ring) $\mathcal{A}_f$} is given by
the quotient ring of the polynomial ring by the Jacobian ideal $J_f$ of $f$.
For example, the Milnor ring of $D_4$ is $\mathbb{C}[x,y]/(y^2+x^2,xy)$.
Note that the Milnor ring for $A_n$ is generated by $\{ 1, z, \dots, z^{n-1}\}$
and for $D_n$ by $\{ 1,x, \dots, x^{n-2}, y\}$ as vector spaces.
Hence their dimensions are equal to $n$.

$\mathcal{A}_f$ has the structure of a Frobenius algebra with counit $\epsilon$ as follows:
\begin{align*}
  A_n :\  \epsilon (z^i) &= \delta_{i,n-1} \quad \text{for } i = 0, \dots, n-1 \ , \\
  D_n :\  \epsilon (x^i) &= \delta_{i, n-2} \quad \text{for } 
  	i = 0, \dots, n-2 \ \quad \text{and} \\
          \epsilon (y) &= 0 
\end{align*}

\subsection{Frobenius Manifold Structures on Milnor Rings}
\label{ssec:fm}
We review the Frobenius manifold structures associated with
the universal unfoldings of the singularities $A_n$ and $D_n$.
For details, see \cite{dvv} and \cite{manin}.

Any vector in the Milnor rings of $A_n$ and $D_n$ can be uniquely written as follows.
\begin{align*}
  A_n &: k_{n-1}z^{n-1} + \dots +k_1 z + k_0 \\
  D_n &: l_{n-2}x^{n-2} + \dots + l_1 x + l_0 + l_\ast y 
\end{align*}
These identify their Milnor rings with $\mathbb{C}^n$,
whose coordinates can be written 
as $(k_{n-1}, \dots, k_0)$ for $A_n$ and $(l_{n-2}, \dots, l_0, l_\ast)$ for $D_n$.
It follows that we have the following identifications of vectors.
\begin{align*}
  A_n &: \partial_{k_i} \leftrightarrow z^i \\
  D_n &: \partial_{l_j} \leftrightarrow x^j \ \text{ and }\  
  \partial_{l_\ast} \leftrightarrow y
\end{align*}

We also associate a deformation of the defining polynomials of $A_n$ and $D_n$
tautologically for each vector in the Milnor ring simply by addition.
\begin{align*}
  A_n &:\frac{1}{n+1} z^{n+1} + k_{n-1}z^{n-1} + \dots +k_1 z + k_0 := F_{A_n}\\
  D_n &:\frac{1}{2} xy^2 + \frac{1}{2n-2}x^{n-1} +
  l_{n-2}x^{n-2} + \dots + l_1 x + l_0 + l_\ast y 
  := F_{D_n}
\end{align*}
We define multiplication rules between vector fields in the way that each tangent
space is isomorphic to the Jacobian ring of the deformed polynomial.
To be more precise, the ring of formal vector fields are isomorphic to the following.
\begin{align*}
  A_n &: \mathbb{C}\llbracket k_0, \dots, k_{n-1}\rrbracket [z]/JF_{A_n} \\
  D_n &: \mathbb{C}\llbracket l_0, \dots, l_{n-2}, l_\ast \rrbracket [x,y]/JF_{D_n} 
\end{align*}
Here $JF_{A_n}$ and $JF_{D_n}$ denote the Jacobian ideal of $F_{A_n}$ and $F_{D_n}$
with respect to $z$ for $A_n$ and $x,y$ for $D_n$.

Next we describe the metrics. 
The metrics are given by the residues.
Their multi-variable version can be found, for example, in \cite{gh}.
We divide by $2$ for $D_n$ for normalization purposes.
Fix a tangent space, take $f,g$ in the Jacobian ring, and let $F$ be
equal to either $F_{A_n}$ or $F_{D_n}$.
$P$ runs through all the poles.
\begin{align*}
  A_n &: \eta(f,g):= \sum_P \res_P \frac{fg}{F'} dz \\ 
  D_n &: \eta(f,g):= \frac{1}{2} \sum_P \res_P 
  \frac{fg}{\frac{\partial F}{\partial x}\frac{\partial F}{\partial y}} dx\ dy
\end{align*}
One can show that the metric at the origin is the same as the one
in the previous section. Thus we can identify the undeformed Milnor ring as
the tangent space at the origin.

Consider the odd dimensional ones in $A_n$ series.
First we identify a series of Frobenius manifolds called $B_{n-1}$ 
as submanifolds in both $A_{2n-3}$ and $D_n$. For generalities about
Frobenius submanifolds, see \cite{stra}.

Consider the $(n-1)$-dimensional submanifold with coordinates $a_i := k_{2i} = 
l_i, i=0, \dots, n-2$, that is, we are considering hypersurfaces given by 
$k_{2i-1} = 0$ and  $l_\ast=0$ . The algebra of tangent vectors at any point
on the submanifold is generated by $z^2$ or $x$. One can show that
the ring homomorphism of tangent spaces 
given by $z^2 \mapsto x$ at each point of the submanifold
is an isomorphism that also preserves the metric. 
 We will use $(a_0, a_1, \dots, a_{2n-4})$ and
$(a_0, a_2, \dots, a_{2n-4}, a_\ast)$ for coordinates of $A_{2n-3}$ and $D_n$, respectively,
so that $(a_0, a_2, \dots, a_{2n-4})$ will be the coordinates for $B_{n-1}$.

Before considering the potential, note that the coordinates we are using are
not flat, meaning that the metric between the coordinate vector fields are not 
constants. We first describe how to get a set of flat coordinates for $A_n$.
If we set
\[ \frac{w^{n+1}}{n+1} = F_{A_n} 
  = \frac{1}{n+1} z^{n+1} + k_{n-1}z^{n-1} + \dots +k_1 z + k_0 \ ,
  \]
then we have the Laurent series expansion
\[ z = w + \frac{t_{n-1}}{w} + \frac{t_{n-2}}{w^2} + \dots + 
  \frac{t_0}{w^n} + O(\frac{1}{w^{n+1}}) \ .
  \]
  It is known that $a_i$ is a polynomial of $t_i, t_{i+2}, \dots, t_{n-1}$ for each $i$,
  where the linear term is equal to $-t_i$ and the higher order terms are functions of
  $t_{i+2}, \dots, t_{n-1}$
  \cite{dvv}.
Thus this gives us a global diffeomorphism, 
and the inverse map can be found easily.
It is also easy to see that the coordinate 
vector fields $\partial_i := \partial/\partial t_i$ can be identified with
$\partial_i F_{A_n}$. Moreover, it can be shown that the metric $\eta_{ij}:=
\eta(\partial_i, \partial_j)$ is constant by change of coordinates,
hence its values can be obtained
by looking at the origin as shown in the previous subsection.

The uniqueness of the series expansion implies that the submanifold $B_{n-1}$
is also the hyperplane defined by the equations $t_{2i-1} = 0$ in the flat coordinates.
Thus $\{t_{2i}\}$ constitutes a flat coordinate system for $B_{n-1}$. Moreover, 
it can be shown that $\{t_{2i}, t_\ast:= - a_\ast\}$ is also a flat coordinate system for
$D_n$.

Now we are ready to describe the potentials. It is known that the potential for
$A_n$ is a polynomial of degree $\leq n+2$ in the flat coordinates. Its explicit
form can be determined from the ring structure and the metric. The potential
for $B_{n-1}$ can be obtained by setting $t_{2i-1} =0$
from the potential for $A_{2n-3}$. 
Finally, it is known that the potential for $D_n$ can be obtained by
adding the term $(-1/2)  a_0 t_\ast^2$ to the $A_{2n-3}$ potential
and setting $t_{2i-1} = 0$, where $a_0$ is the constant term
in the usual coordinates, and is a function of the flat coordinates.

We list some explicit formulas for $n=3$ and $4$.
Let $\Phi_{A_{2n-3}}$ denote the potential for $A_{2n-3}$, etc.

\begin{ex}[$n=3$]
 Flat coordinates and their inverse:
 \begin{align*}
   t_2 &= -a_2 & a_2 &= -t_2 \\
   t_1 &= -a_1 & a_2 &= -t_1 \\
   t_0 &= -a_0 + \frac{1}{2} a_2^2 & a_0 &= -t_0 + \frac{1}{2} t_2^2
 \end{align*}
 \begin{align*}
   \Phi_{A_3} &= -\frac{1}{2}t_0^2t_2 - \frac{1}{2}t_0t_1^2 - \frac{1}{4} t_1^2t_2^2
   - \frac{1}{60}t_2^5\\
   \Phi_{D_3} &= -\frac{1}{2} t_0^2 t_2 + \frac{1}{2} t_0 t_\ast^2 
   - \frac{1}{4} t_2^2 t_\ast^2 -\frac{1}{60} t_2^5
 \end{align*}
 \end{ex}

 \begin{ex}[$n=4$]
  Flat coordinates and their inverse: 
  \begin{align*}
    t_4 &= -a_4 & a_4 &= -t_4\\
    t_3 &= -a_3 & a_3 &= -t_3\\
    t_2 &= -a_2 + \frac{3}{2} a_4^2 & a_2 &= -t_2 + \frac{3}{2} t_4^2\\
    t_1 &= -a_1 + 2 a_3 a_4 & a_1 &= -t_1 + 2 t_3 t_4\\
    t_0 &= -a_0 + \frac{1}{2} a_3^2 +a_2 a_4 -\frac{7}{6} a_4^3 & a_0 &=
    -t_0 + \frac{1}{2} t_3^2 + t_2 t_4 - \frac{1}{3} t_4^3
  \end{align*}
  \begin{align*}
    \Phi_{A_5} &= - \frac{1}{2} t_0^2 t_4 - t_0 t_1 t_3 - \frac{1}{2} t_0 t_2^2
    - \frac{1}{2} t_1^2 t_2 - \frac{1}{4} t_1^2 t_4^2 - t_1 t_2 t_3 t_4  
    - \frac{1}{6} t_1 t_3^3 \\
    &- \frac{1}{6} t_2^3 t_4 - \frac{1}{2} t_2^2 t_3^2
    - \frac{1}{6} t_2^2 t_4^3 - \frac{1}{2} t_2 t_3^2 t_4^2 - \frac{1}{6} t_3^4 t_4
    - \frac{1}{8} t_3^2 t_4^4 - \frac{1}{210} t_4^7\\ 
    \Phi_{D_4} &= - \frac{1}{2} t_0^2 t_4 - \frac{1}{2} t_0 t_2^2 
    + \frac{1}{2} t_0 t_\ast^2
    - \frac{1}{6} t_2^3 t_4
    - \frac{1}{2} t_2 t_4 t_\ast^2
    - \frac{1}{6} t_2^2 t_4^3 
    + \frac{1}{6} t_\ast^2 t_4^3
    - \frac{1}{210} t_4^7
  \end{align*}
 \end{ex}

 \subsection{$\zz$-Frobenius Algebra for $A_{2n-3}$}
 \label{ssec:zzfa}
 We review the $\zz$-Frobenius algebra structure for $A_{2n-3}$ 
 as described in \cite{orbifolding}.
 Call it $H$.
 Let $\I$ be the ideal of $\C[z,y]$ generated by 
 $\left\{ z^{2n-3}, yz, y^2 + z^{2n-4} \right\}$. 
 Then its ring structure is 
 \[
   H = \C[z,y] / \I \, . 
 \]
Note that this is a $2n-2$ dimensional vector space having 
$\left\{ 1, z, \dots, z^{2n-4}, y \right\}$
as a basis.

The $\zz$-graded $\zz$-module structure is the following.
\begin{align*}
  H_i &= \Span \left\{ 1, z^2, \dots, z^{2n-4} \right\} \\
  H_v &= \Span \left\{ z, z^3, \dots, z^{2n-5} \right\} \\
  H_g &= \Span \left\{ y \right\} 
\end{align*}

If we fix an ordered basis for $H$ as 
$\left( 1, z^2, \dots, z^{2n-4}, z, z^3, \dots, z^{2n-5}, y \right)$,
then $\eta$ is given by the following matrix.
\[
\left(
\begin{array}{c|c|c}
  0 \cdots 1 & \raisebox{-15pt}{{\huge\mbox{{$0$}}}} & 0      \\ [-3.5ex]
    \iddots   &                                       & \vdots \\ 
  1 \cdots 0 &                                       & 0      \\ \hline
  \raisebox{-15pt}{{\huge\mbox{{$0$}}}} & 0 \cdots 1  & 0      \\ [-3.5ex]
                                        &   \iddots    & \vdots \\
                                        & 1 \cdots 0  & 0      \\ \hline
  0 \cdots 0 & 0 \cdots 0		             & -1			
\end{array}
\right)
\]

It is straightforward to check that $H$ thus defined is a $\zz$-Frobenius algebra.

\begin{rem}
Note that $\eta$ satisfies the conditions of Theorem~\ref{thm:structure},
and the restriction of $\eta$ on $H_e$ coincides with the metric on
the Frobenius algebra of $A_{2n-3}$.
\label{rem:metric}
\end{rem}

\begin{rem}
  We observe that $H^{\zz}$ is isomorphic to the Frobenius algebra of $D_n$
via the map $z^2 \mapsto x$. Note that it also preserves the metric.
\end{rem}

\subsection{$\mathbb{Z}/2\mathbb{Z}$-Frobenius Manifold for $A_{2n-3}$}
\label{ssec:zzfm}

We use the flat coordinates $\{t_0, \dots, t_{2n-3}, t_\ast\}$
as defined in Section~\ref{ssec:fm} on the underlying $\zz$-graded $\zz$-module of 
the $\zz$-Frobenius algebra of $A_{2n-3}$ as described in the previous section.
Theorem~\ref{thm:structure} gives us a pre-$\mathbb{Z}/2\mathbb{Z}$-Frobenius
manifold structure on it.
This is because we can obtain the potential for $B_{n-1}$ by restricting 
either that of $A_{2n-3}$ or that of $D_n$, and the condition
on the metric is also met by Remark~\ref{rem:metric}.
Since $t_\ast$ always appear as $t_\ast^2$ in the terms of 
$\Phi_{D_n}$, the $G$-degree requirement for $Y$ is also satisfied.
We identify $H_e$ as the Frobenius manifold of $A_{2n-3}$ and
$H^G$ as that of $D_n$.

To see that this is a $\zz$-Frobenius manifold, note that we have
\[ Y^3_g = \frac{1}{2} t_0 t_\ast^2 \] regardless of the value of $n$, from the discussions in
Section~\ref{ssec:fm}.
Considering the metric, 
we notice that this term implies all the multiplication
rule that involves $\partial_\ast \leftrightarrow -y$ (at the origin)
described in Section~\ref{ssec:zzfa}, up to the factor of $1/6$.
Of course, $Y^3_i + Y^3_v$ is just the degree $3$ term of $A_{2n-3}$.
Thus these terms accounts for the multiplication rules for $\partial_i \leftrightarrow -z^i$
(at the origin),
again up to the factor of $1/6$. Hence it contains the same $\zz$-Frobenius algebra
as in Section~\ref{ssec:zzfa}, up to the constant factor of $1/6$.

\section{Prospects}

\subsection{A Differential Geometry of $G$-braided Spaces}
Although we concentrated on the formal aspects of Frobenius manifolds
so far, their theory becomes much richer once we start considering
their differential geometric aspects\cite{du, manin}.
One can then ask if we can extend this part of their theory
to our case. This leads to the question of developing a differential
calculus on the ring of braided multilinear forms.

The usual way of introducing a differential calculus on the
ring of functions in noncommutative geometry is
to define a derivation that satisfies the ordinary
Leibniz rule on the ring.
But in our case it is more natural to define a 
differential operator in terms of multilinear forms.
Namely, the usual rule of differentiation on polynomials
can be transferred into one on symmetric multilinear forms
via their identifications.
Once we extend the same rule to all multilinear forms, then
it turns out that this rule restricted to braided multilinear forms
yields braided multilinear forms.

The problem with this approach is in that 
this definition of differentiation rule need not satisfy
the usual Leibniz rule.
This is not surprising since the usual rule does not seem to take into account
any braided structure.
Namely, the differential operator and the functions freely change their positions,
except possibly gaining the minus sign.
It seems that the usual Leibniz rule lives in symmetric monoidal categories.

Hence it will be interesting to generalize the Leibniz rule so that it would
work in braided monoidal categories.
In fact, one such generalization can be found in \cite[Chapter 10]{majid}
in his version of braided geometry.
We will need a version that would work for our purposes.


\begin{thebibliography}{99}

\bibitem{acv}
  D. Abramovich, A. Corti and A. Vistoli, 
  {\it Twisted bundles and admissible covers}, Comm. Algebra
  {\bf 31} (2003), 3547 - 3618

\bibitem{oqc}
  W. Chen, Y. Ruan, {\it Orbifold Gromov-Witten theory}, Contemp. Math. 
  {\bf 310} (2002), 25-85

\bibitem{croc}
  W. Chen, Y. Ruan, {\it A new cohomology theory of orbifold}, Commun. 
  Math. Phys. {\bf 248} (2004), 1-31

\bibitem{costello}
  K. Costello, {\it Higher genus Gromov-Witten invariants as
  genus zero invariants of symmetric products}, Ann. Math. 
  {\bf 164} (2006), 561-601

 \bibitem{dvv}
  R. Dijkgraaf, E. Verlinde and H. Verlinde, {\it Topological strings in $d < 1$},
  Nucl. Phys. {\bf B 352} (1991), 59-86
  
\bibitem{du}
  B. Dubrovin. {\it Geometry of 2D topological field theories, 
  Integrable systems and
  quantum groups}, Lecture Notes in Math.
  {\bf 1620}, Springer (1996), 120-348

\bibitem{singularity}
  W. Ebeling, {\it Functions of Several Complex Variables and Their Singularities},
  AMS (2007)
  
\bibitem{fg}
  B. Fantechi and L. G\"ottsche, {\it Orbifold cohomology for global 
  quotients}, Duke Math. J.
  {\bf 117} (2003), no. 2, 197-227

  \bibitem{gh}
  P. Griffiths and J. Harris, {\it Principles of Algebraic Geometry},
  John Wiley \& Sons, Inc. (1978)

  \bibitem{jkk}
  T. Jarvis, R. Kaufmann, and T. Kimura, {\it Pointed admissible G-covers and 
  G-cohomological Field  theories}, Compos. Math. {\bf 141} (2005), 926-978

\bibitem{jk}
  T. Jarvis and T. Kimura, {\it Orbifold quantum cohomology of the classifying
  space of a finite group}, in {\it Orbifolds in Mathematics and
  Physics}, Contemp. Math. {\bf 310}, AMS (2002), 123-134

  \bibitem{orbifolding}
  R. Kaufmann, {\it Orbifolding Frobenius algebras}, Internat. J. of Math. 
  {\bf 14} (2003), 573-619

\bibitem{sq}
  R. Kaufmann, {\it Second quantized Frobenius algebras}, Commun. Math.
  Phys. {\bf 248} (2004), 33-83
\bibitem{dt}
  R. Kaufmann, {\it The algebra of discrete torsion}, J. Algebra. {\bf 282}
  (2004), 232-259
  \bibitem{kp} 
  R. Kaufmann and D. Pham, {\it The Drinfel'd double and twisting in stringy
  orbifold theory}, Int. J. Math. {\bf 20} (2009), no. 5, 623-657

  \bibitem{km}
  M. Kontsevich, Y. Manin, {\it Gromov-Witten classes, quantum cohomology,
  and enumerative geometry}, Comm. Math. Phys. {\bf 164} (1994), 525-562  

   \bibitem{majid}
  S. Majid, {\it Foundations of Quantum Group Theory}, Cambridge University
  Press (1995)

\bibitem{manin}
  Y. Manin, {\it Frobenius Manifolds, Quantum Cohomology, and Moduli 
  Spaces}, AMS (1999)

\bibitem{moer}
  I. Moerdijk, {\it Orbifolds as groupoids: an introduction}, 
  in {\it Orbifolds in Mathematics and
  Physics}, Contemp. Math. {\bf 310}, AMS (2002), 205-222 

\bibitem{ms}
  G. Moore and G. Segal, 
  {\it D-branes and K-theory in 2D topological field theory} in
  {\it Dirichlet Branes and Mirror Symmetry}, AMS and Clay Math. Ins. (2009)

\bibitem{serre}
  J-P. Serre, {\it Linear Representations of Finite Groups} trans. L. Scott,
  Springer (1977)

\bibitem{matsu}
  I. Shapiro and T. Matsumura,
  {\it Notes on braided calculus on $G$-graded $G$-modules},
  personal communication.

 \bibitem{stra}
  I. Strachan, {\it Frobenius submanifolds}, J. Geom. Phys. 
  {\bf 38} (2001), 285-307

\bibitem{turaev}
  V. Turaev, {\it Homotopy Quantum Field Theory}, EMS (2010)
  
  \end{thebibliography}
\end{document}